\documentclass[leqno]{prims}
\usepackage{hyperref}

\usepackage{enumerate}

\usepackage{amsmath,mathtools,amssymb,extarrows,mathrsfs}
\usepackage{tikz-cd}
\usepackage{verbatim}
\tikzcdset{every label/.append style = {font = \small}} 
\usepackage{xcolor}
\usepackage[all,cmtip]{xy}

\newtheorem{thm}{Theorem}[section]
\newtheorem{cor}[thm]{Corollary}
\newtheorem{lemma}[thm]{Lemma}
\newtheorem{prop}[thm]{Proposition}

\theoremstyle{definition}
\newtheorem{defi}[thm]{Definition}
\newtheorem{rem}[thm]{Remark}

\newtheorem{lemmadefi}[thm]{Lemma-Definition}

\numberwithin{equation}{section}

\newcommand{\D}{\mathbb{D}}
\newcommand{\C}{\mathbb{C}}

\newcommand{\PP}{\mathbb{P}}
\newcommand{\R}{\mathbb{R}}
\newcommand{\Z}{\mathbb{Z}}
\newcommand{\EE}{\mathrm{E}}
\newcommand{\LL}{\mathrm{L}}
\newcommand{\RR}{\mathrm{R}}
\newcommand{\cA}{\mathcal{A}}
\newcommand{\cD}{\mathcal{D}}
\newcommand{\cE}{\mathcal{E}}

\newcommand{\cH}{\mathcal{H}}
\newcommand{\cL}{\mathcal{L}}
\newcommand{\cM}{\mathcal{M}}
\newcommand{\cN}{\mathcal{N}}

\newcommand{\cO}{\mathcal{O}}
\newcommand{\cX}{\mathcal{X}}

\newcommand{\Mod}[1]{\mathrm{Mod}(#1)}
\newcommand{\indlim}[1]{\mathop{``\varinjlim"}\limits_{#1}}
\newcommand{\DbholD}[1]{\mathrm{D}^\mathrm{b}_\mathrm{hol}(\mathcal{D}_{#1})}
\newcommand{\ModD}[1]{\mathrm{Mod}(\mathcal{D}_{#1})}
\newcommand{\DbC}[1]{\mathrm{D}^\mathrm{b}(\C_{#1})}
\newcommand{\DbRcC}[1]{\mathrm{D}^\mathrm{b}_{\mathbb{R}\text{-}\mathrm{c}}(\C_{#1})}

\newcommand{\DbRck}[1]{\mathrm{D}^\mathrm{b}_{\mathbb{R}\text{-}\mathrm{c}}(k_{#1})}
\newcommand{\DbCck}[1]{\mathrm{D}^\mathrm{b}_{\mathbb{C}\text{-}\mathrm{c}}(k_{#1})}
\newcommand{\ModRcC}[1]{\mathrm{Mod}_{\mathbb{R}\text{-}\mathrm{c}}(\C_{#1})}

\newcommand{\ModRck}[1]{\mathrm{Mod}_{\mathbb{R}\text{-}\mathrm{c}}(k_{#1})}
\newcommand{\ModCck}[1]{\mathrm{Mod}_{\mathbb{C}\text{-}\mathrm{c}}(k_{#1})}

\newcommand{\EbIC}[1]{\mathrm{E}^\mathrm{b}(\mathrm{I}\C_{#1})}
\newcommand{\EbRcIC}[1]{\mathrm{E}^\mathrm{b}_{\mathbb{R}\textnormal{-}\mathrm{c}}(\mathrm{I}\C_{#1})}
\newcommand{\Dbk}[1]{\mathrm{D}^\mathrm{b}(k_{#1})}
\newcommand{\DbIk}[1]{\mathrm{D}^\mathrm{b}(\mathrm{I}k_{#1})}
\newcommand{\EbIk}[1]{\mathrm{E}^\mathrm{b}(\mathrm{I}k_{#1})}
\newcommand{\EbRcIk}[1]{\mathrm{E}^\mathrm{b}_{\mathbb{R}\textnormal{-}\mathrm{c}}(\mathrm{I}k_{#1})}

\newcommand{\Amod}{\cA^\mathrm{mod}_{\widetilde{X}}}
\newcommand{\Ard}{\cA^{\mathrm{rd}}_{\widetilde{X}}}
\newcommand{\AmodD}{\cA^\mathrm{mod}_{\XD}}
\newcommand{\Amodf}{\cA^\mathrm{mod}_{\Xf}}
\newcommand{\Ardf}{\cA_{\Xf}^{\mathrm{rd}}}

\newcommand{\Ot}[1]{\mathcal{O}^\mathrm{t}_{#1}}
\newcommand{\DR}{\mathrm{DR}_X}

\newcommand{\DRbuD}{\mathrm{DR}_{\XD}}

\newcommand{\DRE}{\mathrm{DR}^\EE_X}

\newcommand{\DREbuD}{\mathrm{DR}^\EE_{\XD}}
\newcommand{\DREbuf}{\mathrm{DR}^\EE_{\Xf}}
\newcommand{\DRmod}{\mathrm{DR}^\mathrm{mod}_{\widetilde{X}}}
\newcommand{\DRmodD}{\mathrm{DR}^\mathrm{mod}_{\XD}}
\newcommand{\DRmodf}{\mathrm{DR}^\mathrm{mod}_{\Xf}}
\newcommand{\DRrd}{\mathrm{DR}^\mathrm{rd}_{\widetilde{X}}}
\newcommand{\DRrdD}{\mathrm{DR}^\mathrm{rd}_{\widetilde{X}_D}}
\newcommand{\DRrdf}{\mathrm{DR}^\mathrm{rd}_{\widetilde{X}_f}}
\newcommand{\SolE}{\mathcal{S}ol^\EE_X}

\newcommand{\SolEbuD}{\mathcal{S}ol^\EE_{\widetilde{X}_D}}

\newcommand{\sh}[1]{\mathsf{sh}_{#1}}

\newcommand{\shbu}{\mathsf{sh}_{\widetilde{X}}}
\newcommand{\shbuD}{\mathsf{sh}_{\widetilde{X}_D}}
\newcommand{\shbuf}{\mathsf{sh}_{\widetilde{X}_f}}

\newcommand{\RIhom}{\mathrm{R}\mathcal{I}hom}
\newcommand{\RHomE}{\RR\mathcal{H}om^\EE}
\DeclareMathOperator{\Hom}{RHom}
\newcommand{\sHom}{\textnormal{R}\mathcal{H}om}
\newcommand{\DE}{\mathrm{D}^\EE_X}
\newcommand{\hyp}{\mathbb{H}}
\newcommand{\HdR}{\mathrm{H}_{\mathrm{dR}}}
\newcommand{\Hrd}{\mathrm{H}^{\mathrm{rd}}}

\newcommand{\RGamma}{\textnormal{R}\Gamma}

\newcommand{\XD}{\wt{X}_D}
\newcommand{\Xf}{\wt{X}_f}

\newcommand{\defeq}{\vcentcolon=}
\newcommand{\wt}[1]{\widetilde{#1}}
\newcommand{\iso}{\simeq}
\newcommand{\conv}{\overset{+}{\otimes}}

\newcommand{\To}{\longrightarrow}
\newcommand{\ToPO}{\overset{+1}{\longrightarrow}}

\newcommand{\id}{\mathrm{id}}

\newcommand{\fj}{\mbox{$\tilde{j}_f$}}
\newcommand{\gj}{\mbox{$\tilde{j}_g$}}
\newcommand{\Dj}{\mbox{$\tilde{j}_D$}}
\newcommand{\tj}{{\tilde{j}}}
\newcommand{\Xtot}{\widetilde{X}^{\mathrm{tot}}_D}
\newcommand{\itot}{i^{\mathrm{tot}}}
\newcommand{\ttau}{\mbox{$\wt{\tau}$}}
\newcommand{\ttaug}{\mbox{$\wt{\tau}_g$}}
\newcommand{\tpzero}{\mbox{$\wt{p}_0$}}

\newcommand{\FN}{\mathrm{FN}}
\newcommand{\DFN}{\mathrm{DFN}}

\newcommand{\wtens}{\overset{\mathrm{w}}{\otimes}}
\newcommand{\Ltens}{\overset{\mathrm{L}}{\otimes}}
\newcommand{\Cinf}[1]{\mathscr{C}^\infty_{#1}}
\newcommand{\Dbt}[1]{\mathcal{D}b^{\mathrm{t}}_{#1}}
\newcommand{\Dbtv}[1]{\mathcal{D}b^{\mathrm{t\vee}}_{#1}}
\newcommand{\varpitot}{\varpi^{\mathrm{tot}}}

\VolumeNo{4x}
\YearNo{201x}
\communication{Communicated by T. Mochizuki. Received December 24, 2022. Revised March 23, 2023; June 29, 2023.}

\begin{document}



\TitleHead{Moderate Growth and Rapid Decay Cycles}
\title{Moderate Growth and Rapid Decay Nearby Cycles via Enhanced Ind-Sheaves}


\AuthorHead{Hepler and Hohl}
\author{Brian \textsc{Hepler}\footnote{Hepler: University of Wisconsin-Madison, 480 Lincoln Drive, Madison, Wisconsin 53706, USA;
\email{bhepler@wisc.edu}} and
Andreas \textsc{Hohl}\footnote{Hohl: Université Paris Cité and Sorbonne Université, CNRS, IMJ-PRG (Institut de Mathématiques de Jussieu-Paris Rive Gauche), F-75013 Paris, France; \email{andreas.hohl@imj-prg.fr}}
}


\maketitle

\begin{abstract}
For any holomorphic function $f\colon X\to\C$ on a complex manifold $X$, we define and study moderate growth and rapid decay objects associated to an enhanced ind-sheaf on $X$. These will be sheaves on the real oriented blow-up space of $X$ along $f$. We show that, in the context of the irregular Riemann--Hilbert correspondence of D'Agnolo--Kashiwara, these objects recover the classical de Rham complexes with moderate growth and rapid decay associated to a holonomic $\cD_X$-module. 

    In order to prove the latter, we resolve a recent conjectural duality of Sabbah between these de Rham complexes of holonomic $\cD_X$-modules with growth conditions along a normal crossing divisor by making the connection with a classic duality result of Kashiwara--Schapira between certain topological vector spaces. Via a standard d\'evissage argument, we then prove Sabbah's conjecture for arbitrary divisors. As a corollary, we recover the well-known perfect pairing between the algebraic de Rham cohomology and rapid decay homology associated to integrable connections on smooth varieties due to Bloch--Esnault and Hien.
\end{abstract}

\tableofcontents      

\section{Introduction}\label{sec:intro}

The theories of ind-sheaves (initiated in \cite{KS01}) and enhanced ind-sheaves (established in \cite{DK16}) led to an extension of the classical Riemann--Hilbert correspondence for regular holonomic $\cD$-modules established by M.\ Kashiwara in \cite{Kas80}, \cite{Kas84} (see also Z.\ Mebkhout \cite{Meb84} for a different proof). This correspondence states that the de Rham functor
$$\DR\colon \mathrm{D}^\mathrm{b}_{\mathrm{rh}}(\cD_X)\overset{\sim}{\longrightarrow} \mathrm{D}^\mathrm{b}_{\C\text{-}\mathrm{c}}(\C_X)$$
from the derived category of regular holonomic D-modules to the derived category of $\C$-constructible sheaves on a complex manifold $X$ is an equivalence of derived categories.

It was not difficult to observe that this functor is no longer fully faithful on the larger category of (not necessarily regular) holonomic $\cD_X$-modules; the simplest examples of this failure come from the so-called \emph{exponential $\cD_X$-modules} with poles along a divisor $D \subset X$. These objects are of the form $\mathcal{E}_{X\setminus D | X}^\varphi \vcentcolon= (\cO_X(*D),d-d\varphi)$, where the ``exponent" $\varphi \in \cO_X(*D)/\cO_X$ is an (unbounded) meromorphic function with poles along $D$, or more generally, a Puiseux germ defined in some sectorial neighborhood along $D$.

Finding a suitable target category for this conjectural ``irregular" Riemann--Hilbert correspondence was a long-standing program that led to the development of the above-mentioned theories and the following result due to Andrea D'Agnolo and Masaki Kashiwara:

\begin{thm}[{\cite[Theorem 9.5.3]{DK16}}]
Let $X$ be a complex manifold. Then, there is a fully faithful functor 
$$
\DRE : \DbholD{X} \hookrightarrow \EbRcIC{X},
$$
called the \emph{enhanced de Rham functor}, from the derived category of holonomic $\cD_X$-modules to the triangulated tensor category of $\R$-constructible \emph{enhanced ind-sheaves} on $X$.\footnote{In fact, one can restrict the target category further: The essential image is precisely the category of $\C$-constructible enhanced ind-sheaves, a notion introduced in \cite{Ito20} (see also \cite{Kuw21} for an alternative approach).  In this way, the functor $\DRE$ becomes an equivalence of categories, but we will not use these notions here.}
\end{thm}

From a general point of view, this result teaches us that constructible enhanced ind-sheaves should be our objects of interest in the study of holonomic $\cD$-modules with irregular singularities---just as we were interested in constructible sheaves as the topological counterparts of regular holonomic $\cD$-modules. One of the purposes of this paper is therefore to introduce objects similar to the ones known from the classical theory of constructible sheaves.

Classical objects in the study of differential equations with irregular singularities are the de Rham complexes with moderate growth and rapid decay. 
The reason why such complexes are important is the following: there are non-isomorphic $\cD$-modules whose (classical) de Rham complexes are isomorphic. Roughly speaking, this happens because their sheaves of holomorphic solutions are the same, even though the growth behavior of their solutions is different, but the classical de Rham functor is not sensitive to growth conditions. In order to obtain a functor which can distinguish such $\cD$-modules, it is therefore necessary to introduce variants of the de Rham functor that can ``measure'' the growth of the solutions to a differential system. A technical difficulty in these constructions is the need to work on \emph{real} blow-up spaces, where functions with moderate growth and rapid decay along the exceptional divisor are more easily studied. (Let us note that the use of the real blow-up can, however, partly be avoided by working with ind-sheaves.)

\bigskip

In the one-dimensional case, a solution to the irregular Riemann--Hilbert problem has been described earlier by introducing the notion of \emph{Stokes structure}, which can be formulated in several different ways, one of them being \emph{Stokes-filtered local systems} (see \cite{Mal91}, and we refer to \cite{Sab13} for an exposition on Stokes filtrations for $\cD$-modules, and to \cite{Boa21} for a history and comparison of different formulations of the Stokes phenomenon). Indeed, the Stokes filtration on the local system associated to a holonomic $\cD$-module $\cM$ is essentially given by the moderate growth de Rham complexes of ``exponentially twisted" versions of $\cM$.

\bigskip

The simple idea at the heart of this article is now the following:
Since $\DRE$ is fully faithful, all the information about a holonomic $\cD_X$-module must be encoded in its enhanced de Rham complex. In particular, there should be a functorial way to obtain the moderate growth and rapid decay de Rham complexes of a holonomic $\cD_X$-module $\cM$ from the ($\R$-constructible) enhanced ind-sheaf $\DRE(\cM)$, without leaving the topological setting. Our goal is to construct, for any holomorphic function $f\colon X \to \C$, functors $(-)^{\mathrm{mod} \,  f}$ and $(-)^{\mathrm{rd}\, f}$ that make the following diagram commutative: 
\begin{equation}\label{eq:diagramMainQuestion}
\begin{tikzcd}
\DbholD{X}\arrow{rr}{\mathrm{DR}^\mathrm{E}_X} \arrow{rd}[swap]{\mathrm{DR}^{\mathrm{mod}}_{\widetilde{X}_f}, \DRrdf} && \EbRcIC{X} \arrow{ld}{(-)^{\mathrm{mod}\, f}, (-)^{\mathrm{rd}\,f}}\\
&\DbRcC{\Xf}
\end{tikzcd}\tag{$\star$}
\end{equation}
where $\Xf$ denotes the real blow-up of $X$ along $f$. In particular, in order to do this, we will investigate two notions that do not seem to have been studied in the context of enhanced ind-sheaves yet: 
\begin{enumerate}
\item We develop the notion of the enhanced de Rham complex \emph{on the real blow-up along a holomorphic function $f$} (in \cite{DK16} and \cite{KS16}, only the real blow-up along a normal crossing divisor has been studied).

\item We study the notions of rapid decay functions and the rapid decay de Rham complexes using enhanced ind-sheaves. (In \cite{DK16} and \cite{KS16}, only the sheaf $\AmodD$ of holomorphic functions with moderate growth along a normal crossing divisor has been studied). 

\end{enumerate}
In order to achieve results about rapid decay objects on these differing notions of real blow-up, we will mimic constructions done in \cite{DK16} in the case of moderate growth along a normal crossing divisor, and additionally make a connection between moderate growth and rapid decay in Section \ref{sec:TWduality}. This is accomplished using a deep categorical duality of Kashiwara--Schapira between the related notions of tempered distributions and Whitney functions (see \cite[Theorem~6.1]{KS96} for the original duality statement; additionally, see Subsection 2.5.4 of the recent survey \cite{KS16} for a more expository treatment). 
In this way, we finally prove (in Corollary~\ref{cor:dualityIsoXfNCD}, Lemma~\ref{conj:natMorph}, and Proposition~\ref{prop:conjNotNCD}) the following duality conjectured by Claude Sabbah (see Theorem~\ref{thm:SabDuality}):

\begin{thm}[{cf.\ \cite[Conjecture 4.13]{Sab21}}]\label{conj:SabDuality}
Let $X$ be a complex manifold, $\cM$ a holonomic $\cD_X$-module, and $f\colon X \to \C$ a holomorphic function on $X$. Let $\Xf$ denote the real blow-up of $X$ along $f$. Then, there is a natural isomorphism in $\DbRcC{\Xf}$
$$
\DRrdf(\mathbb{D}_X\cM) \xlongrightarrow{\thicksim} \mathrm{D}_{\Xf}\DRmodf(\cM),
$$
where $\mathbb{D}_X$ denotes the $\cD_X$-module duality functor on $\DbholD{X}$, and $\mathrm{D}_{\Xf}$ denotes the Verdier duality functor on $\DbC{\Xf}$.
\end{thm}

The key to its proof is the understanding of the case of a normal crossing divisor, which we treat first (see Proposition~\ref{prop:conjSabNCD} and Corollary~\ref{cor:dualityIsoXfNCD}) and which makes use of Kashiwara--Schapira's duality. The general case then follows using resolution of singularities, once a canonical morphism has been constructed (see Lemma~\ref{conj:natMorph} and Proposition~\ref{prop:conjNotNCD}).

In particular, this theorem will also make clear the connection between the above-mentioned duality of Kashiwara--Schapira and the now classical dualities coming from the period pairings of Bloch--Esnault \cite{BE04} and Hien \cite{Hie09}, on which we will elaborate at the end of Section~\ref{subsec:dualityBEH}.

The main result of this article, giving an answer to the question of diagram~\eqref{eq:diagramMainQuestion}, is summarized in the following theorem.
\begin{thm}
    To an object $K\in\EbIk{X}$, we functorially associate the objects of $\Dbk{\Xf}$
    \begin{align*}
    K^{\mathrm{mod}\, f} \vcentcolon= \shbuf(\EE \fj_* \EE j^{-1} K),\\
    K^{\mathrm{rd}\, f} \vcentcolon= \shbuf(\EE \fj_{!!} \EE j^{-1} K),
    \end{align*}
    where $j\colon (X\setminus f^{-1}(0))_\infty \hookrightarrow X$ and $j_f\colon (X\setminus f^{-1}(0))_\infty\hookrightarrow \Xf$ are the inclusions of the divisor's complement (in the sense of bordered spaces).
    
    The functors $(\bullet)^{\mathrm{mod}\, f}$ and $(\bullet)^{\mathrm{rd}\, f}$ preserve $\R$-constructiblity, and if $k=\C$ and $K=\DRE(\cM)$ for some $\cM\in\DbholD{X}$, then we have isomorphisms
    \begin{align*}
        K^{\mathrm{mod}\, f}\iso \DRmodf(\cM),\\
        K^{\mathrm{rd}\, f}\iso \DRrdf(\cM).
    \end{align*}
\end{thm}
The constructibility result is immediate (see Lemma~\ref{lemma:constructible}). On the other hand, as mentioned above, the idea of proof for the rest of the remaining statements goes as follows: We first consider the case of a normal crossing divisor. The result about moderate growth is then rather directly deduced from constructions similar to those in \cite{DK16} (see Proposition~\ref{prop:DRmod}), while the rapid decay case follows from the moderate growth case and Theorem~\ref{conj:SabDuality} (see Proposition~\ref{prop:conjSabNCD}). In addition, in both cases, some work is necessary to deduce the general case from that of a normal crossing divisor (see Propositions~\ref{prop:DRmodf} and \ref{prop:conjNotNCD}).

\subsection*{Outline of the paper}

Section~\ref{sec:review} is a brief review of the languages of sheaves, ind-sheaves, enhanced ind-sheaves and $\cD$-modules, as well as the theory of real oriented blow-up spaces, all of which we will use in this paper. We contextualize all of this background language in Section~\ref{sec:motivation}, where we give an overview of how these objects interact in dimension one via the Riemann--Hilbert correspondence. Afterward, we introduce and motivate the duality results later proved in Section~\ref{sec:duality} with some concrete computations on exponential $\cD$-modules. 

In Section~\ref{sec:DRmodandrd}, we motivate our later definition of moderate growth and rapid decay objects associated to enhanced ind-sheaves: First, we show how the moderate growth de Rham complex of a holonomic $\cD_X$-module $\cM$ along a divisor can be recovered in a functorial way from its enhanced de Rham complex $\DRE(\cM)$ (Proposition \ref{prop:DRmod}). This is greatly inspired by the formulas developed in \cite{DK20cycles}, which are briefly reviewed at the beginning of the section. The proof in the higher-dimensional case works in two steps: We prove the result in the case of a simple normal crossing divisor, and then use results about resolution of singularities to prove a statement in the general case. We then motivate a similar functorial construction for the rapid decay de Rham complex of $\cM$, which will be proved in the case of normal crossing divisor in Section~\ref{sec:TWduality}. 

In Section~\ref{sec:enhmodrd}, we define functors that associate to an enhanced ind-sheaf moderate growth and rapid decay objects in the (classical) category of sheaves on the real blow-up. These naturally lead to definitions for moderate growth and rapid decay nearby cycles in the category of sheaves on the boundary of the real blow-up space. These definitions will be motivated by the statements in the preceding section, so that these objects will, in particular,recover moderate growth de Rham complexes of holonomic $\cD$-modules. We also prove some elementary properties of these objects, similar to those known from classical nearby cycles (see \cite{KS90}) and moderate growth and rapid decay objects associated to holonomic $\cD$-modules (cf.\ Section 6 of \cite{Sab21}).

Restricting ourselves to $\R$-constructible enhanced ind-sheaves, the above objects with growth conditions are trivially seen to be related by Verdier duality. As a consequence, there is a natural duality pairing between such objects in the derived category $\DbRcC{{\Xf}}$ (see Section~\ref{sec:duality}). This observation forms the basis of our proof of Sabbah's Conjecture \ref{conj:SabDuality}.

In Section~\ref{sec:TWduality}, after reviewing a duality result of \cite{KS96} and \cite{KS16} between tempered distributions and Whitney functions, we first prove Conjecture~\ref{conj:SabDuality} in the case of a simple normal crossing divisor (Proposition~\ref{prop:conjSabNCD}, Corollary~\ref{cor:dualityIsoXfNCD}) and show that duality interchanges the moderate growth de Rham complex with rapid decay de Rham complex for holonomic $\cD$-modules. Afterward, we are able to use a standard d\'evissage argument via resolution of singularities to reduce the general case to that of a simple normal crossing divisor. 

This result has two interesting implications: First, it shows that the rapid decay object defined in Section~\ref{sec:enhmodrd} does indeed recover the rapid decay de Rham complex of a holonomic $\cD$-module (Corollary~\ref{cor:DRrdXD}). Second, our result recovers the natural pairing between rapid decay homology and algebraic de Rham cohomology due to Bloch--Esnault \cite{BE04} and M.\ Hien \cite{Hie09} (see Proposition~\ref{cor:ModRdpairing1}).

\section{Background and notation}\label{sec:review}

\subsection{From sheaves to enhanced ind-sheaves}
In this section, we will recall some basic notation in the context of sheaf theory and its generalizations. Let $X$ be a topological space (all topological spaces will be assumed to be \emph{good}, i.e., Hausdorff, locally compact,
second countable and having finite flabby dimension). Let $k$ be a field. 

\subsubsection*{Sheaves} We denote by $\Mod{k_X}$ the category of sheaves of $k$-vector spaces on $X$, and by $\Dbk{X}$ its bounded derived category. One has the six Grothendieck operations $\RR\cH om, \otimes, \RR f_*, f^{-1}, \RR f_!, f^!$ (for $f$ a morphism of good topological spaces) on $\Dbk{X}$. We denote the (Verdier) duality functor by $\mathrm{D}_{X}=\RR\cH om_{k_X}(\bullet,\omega_X)$, where $\omega_X$ is the (Verdier) dualizing complex. Recall that $\omega_X \vcentcolon= a_X^!k$, where $a_X : X \to \{pt\}$ is the natural map. For a locally closed set $Z\subseteq X$, we will denote by $k_Z\in\Mod{k_X}$ the constant sheaf with stalk $k$ on $Z$, extended by zero outside $Z$. We refer to the standard literature, e.g., Chapters 2 and 3 of \cite{KS90}, for details on sheaf theory.

\subsubsection*{Ind-sheaves} We denote by $\mathrm{I}(k_X)$ the category of ind-sheaves over $k$ on $X$, which has been constructed in \cite{KS01} as the category of ind-objects for the category of compactly supported sheaves of $k$-vector spaces. The inductive limit in $\mathrm{I}(k_X)$ will be denoted by $\indlim{}$. We denote by $\DbIk{X}$ the bounded derived category of $\mathrm{I}(k_X)$. One has the six Grothendieck operations $\RIhom, \otimes, \RR f_*, f^{-1}, \RR f_{!!}, f^!$ on it. 

\begin{rem}\label{rem:indAdjunctions}
Moreover, one has a fully faithful embedding $\iota_X\colon \Mod{k_X}\to\mathrm{I}(k_X)$. It has a left adjoint $\alpha_X$, which itself has a fully faithful left adjoint $\beta_X$. Let us note that the three functors $\iota_X$, $\alpha_X$ and $\beta_X$ are exact. Let us also note that the natural inclusion $\iota_X$ is mostly suppressed in the notation (starting with \cite{KS01}) and that also the functor $\beta_X$ is often suppressed in the more recent notational conventions of \cite{DK16} and \cite{KS16} (and works thereafter). We refer to \cite{KS01} for details on the theory of ind-sheaves.
\end{rem}

\subsubsection*{Enhanced ind-sheaves} The category of enhanced ind-sheaves has been constructed in \cite{DK16} in order to establish a Riemann--Hilbert correspondence for holonomic $\cD$-modules. 

To this end, the authors of loc.~cit.\ introduced the notion of a \emph{bordered space}, which is a pair $(X, \widehat{X})$ of good topological spaces such that $X\subseteq \widehat{X}$ is an open subspace. A morphism of bordered spaces $(X,\widehat{X})\to (Y,\widehat{Y})$ is a continuous map $X\to Y$ such that, denoting by $\overline{\Gamma}$ the closure of its graph in $\widehat{X}\times\widehat{Y}$, the map $\overline{\Gamma}\to \widehat{X}$ induced by the projection to the first factor is proper. (Note that this condition is in particular satisfied if $\widehat{Y}$ is compact.) It follows that the category of good topological spaces is naturally a subcategory of the category of bordered spaces by considering a topological space $X$ as the pair $(X,X)$.

Let $\cX=(X,\widehat{X})$ be a bordered space, and consider also the bordered space $\R_\infty\defeq(\R,\mathsf{P})$, where $\mathsf{P}\defeq \PP^1(\R)$ is the real projective line. One then defines the quotient categories
$$\mathrm{D}^\mathrm{b}(\mathrm{I}k_{\cX\times\R_\infty})\defeq \mathrm{D}^\mathrm{b}(\mathrm{I}k_{\widehat{X}\times\mathsf{P}})/\mathrm{D}^\mathrm{b}(\mathrm{I}k_{(\widehat{X}\times\mathsf{P})\setminus(X\times\R)})$$
and
$$\EbIk{\cX}\defeq \mathrm{D}^\mathrm{b}(\mathrm{I}k_{\cX\times\R_\infty})/\pi^{-1}\DbIk{\cX},$$
where $\pi=\pi_{\cX}\colon \cX\times\R_\infty\to\cX$ denotes the projection. We refer to \cite{DK16} for details on this construction and to \cite{KS06} for an exposition on quotients of triangulated categories. One calls $\EbIk{\cX}$ the category of \emph{enhanced ind-sheaves} over $k$ on $\cX$. Throughout this paper, we will assume $\cX$ is of the form $(X,X)$ for some good space $X$ and just write $\cX = X$ unless otherwise indicated.

It admits the six Grothendieck operations $\RIhom^+, \conv, \EE f_*, \EE f^{-1}, \EE f_{!!}$ and $\EE f^!$ for a morphism of bordered spaces $f$. Moreover, we denote by $\mathrm{D}^\EE_{X}$ the duality functor for enhanced ind-sheaves.

One of the most important objects in the category $\EbIk{X}$ is the \emph{enhanced constant sheaf}
$$k^\EE_X\defeq \indlim{a\to\infty} k_{\{t\geq a\}},$$
where we abbreviate $\{t\geq a\}\defeq \{(x,t)\in\widehat{X}\times\mathsf{P}; x\in X, t\in\R, t\geq a\}$.

For a bordered space $\cX=(X,\widehat{X})$, there is a fully faithful embedding
\begin{align*}
    e_\cX\colon \Dbk{X}&\hookrightarrow \EbIk{\cX},\\
    F&\mapsto k^\EE_\cX \otimes \pi^{-1}F.
\end{align*}

In \cite{DK21sh}, the authors introduced the \emph{sheafification functor} $\sh{X}\colon \EbIk{X}\to \Dbk{X}$ for enhanced ind-sheaves on $X$. In the case of (in particular for $\R$-constructible, see below) enhanced ind-sheaves, it is given by
$$\sh{\cX}(K)\defeq \RR\cH om^\EE(k^\EE_{\cX},K)\in\Dbk{X}.$$
(For a definition of the functor $\sHom^\EE$, we refer to \cite[Definition 4.5.13]{DK16}, where it was denoted by $\cH om^\EE$. See also \cite{DK21sh} for a detailed study of the sheafification functor.)
Additionally, there is a natural isomorphism $\sh{X}\circ e_X\iso\id_{\Dbk{X}}$, i.e., $\sh{\cX}$ is a left quasi-inverse of $e_{X}$.

\subsubsection*{Constructibility}
In classical sheaf theory, different constructibility conditions for sheaves of vector spaces have been studied (see, e.g., Chapter 8 of \cite{KS90}).
In particular, if $X$ is a real analytic manifold (or more generally, a \emph{subanalytic space}, see, e.g., \cite[Exercise 9.2]{KS90}), there are  thick subcategories $\ModRck{X}\subset\Mod{k_X}$ of $\R$-constructible sheaves and $\DbRck{X}\subset \Dbk{X}$ of complexes with $\R$-constructible cohomologies.

If $X$ is a complex manifold, a stronger notion is that of a $\C$-constructible sheaf, and one has thick subcategories $\ModCck{X}\subset \mathrm{Mod}(k_{X})$ and $\DbCck{X}\subset \Dbk{X}$.

For enhanced ind-sheaves on a real analytic bordered space $\cX=(X,\widehat{X})$ (meaning that $X$ and $\widehat{X}$ are real analytic manifolds), there is an analogous thick subcategory $\EbRcIk{\cX}\subset \EbIk{\cX}$ of $\R$-constructible enhanced ind-sheaves (see \cite[§4.9]{DK16} for details).

A notion of $\C$-constructibility has been studied in \cite{Ito20} (see also \cite{Kuw21} and \cite{Moc22a,Moc22b} for different approaches to describe the essential image of the Riemann--Hilbert functor of \cite{DK16}).

\subsection{$\cD$-modules}
When $X$ is a complex manifold, we will denote by $\cD_X$ the sheaf of linear partial differential operators with holomorphic coefficients on $X$. We refer to the standard literature, such as \cite{Kas03, HTT08}, for the theory of $\cD$-modules (on complex manifolds as well as on smooth algebraic varieties). For a morphism of complex manifolds $f\colon X\to Y$, the direct and inverse image operations for $\cD$-modules will be denoted by $\mathrm{D}f_*$ and $\mathrm{D}f^*$, respectively.

Let $X$ be a complex manifold.
We denote by $\mathrm{Mod}_{\mathrm{hol}}(\cD_X)$ the category of holonomic $\cD_X$-modules and by $\DbholD{X}$ the subcategory of the derived category of $\cD_X$-modules consisting of complexes with holonomic cohomologies. 

The \emph{(classical) de Rham functor} is given by
\begin{align*}
 \DR\colon \DbholD{X}&\longrightarrow \DbC{X}, \\ \cM &\longmapsto \Omega_{X}\Ltens_{\cD_X} \cM,
\end{align*}
where $\Omega_X \vcentcolon= \Omega_X^{d_X}$ is the invertible sheaf of top-degree holomorphic differential forms on $X$, where $d_X \vcentcolon= \dim_\C X$. $\Omega_X$ is also a right $\cD_X$-module, and is used to define the duality functor for holonomic left $\cD_X$-modules $\cM \in \DbholD{X}$:
\begin{align*}
\mathbb{D}_X\cM \vcentcolon= \sHom_{\cD_X}(\cM,\cD_X) \otimes_{\cO_X} \Omega_X^{\otimes-1}[d_X],
\end{align*}
which is another holonomic left $\cD_X$-module.

In \cite{DK16}, D'Agnolo--Kashiwara define the \emph{enhanced de Rham functor} 
\begin{align*}
 \DRE\colon \DbholD{X}&\longrightarrow \EbIC{X}, \\ \cM &\longmapsto \Omega_{X}^\EE \Ltens_{\cD_X} \cM 
\end{align*}
where $\Omega_{X}^\EE \vcentcolon=\Omega_X\Ltens_{\cO_X} \cO^\EE_X$, and $\cO^\EE_X$ is the enhanced ind-sheaf of tempered holomorphic functions. 

\begin{thm}[{cf.\ \cite[Theorem 9.5.3]{DK16}}]
The functor $\DRE$ is fully faithful.
\end{thm}

The functors $e_X$ and $\sh{X}$ make a connection between the classical and the enhanced de Rham functor: For a regular holonomic $\cD_X$-module $\mathcal{R}\in\mathrm{D}^\mathrm{b}_{\mathrm{rh}}(\cD_X)$, one has 
$$\DRE(\cM)\iso e_X\DR(\cM).$$
On the other hand, for any holonomic $\cD_X$-module $\cM\in\DbholD{X}$, one has
$$\sh{X}\DRE(\cM)\iso \DR(\cM).$$

\subsection{Real blow-up spaces}\label{subsec:realblowup} In the following, we will recall two (in general different) constructions of a real oriented blow-up space (often simply called the \emph{real blow-up}) associated to a complex manifold and a divisor. We will also recall some important sheaves on these blow-up spaces. For more details, we refer to Section 2 of \cite{Sab21}, Subsection 4.1.5 of \cite{Moc14}, Subsection 7.1 of \cite{DK16}, or Subsection 4.2 of \cite{KS16}.

\subsubsection*{Real blow-up along a function}
Let $X$ be a complex manifold and $f\colon X\to \C$ a holomorphic function. Then the \emph{real blow-up of $X$ along $f$} is denoted by $\varpi_f\colon \Xf \to X$ and defined as follows: Consider the map $f/|f|\colon X \setminus f^{-1}(0) \to S^1$, then $\Xf$ is the closure of its graph in $X\times S^1$. The map $\varpi_f\colon \Xf\to X$ is induced by the projection to the first factor. It is a homeomorphism on $X^*\vcentcolon= X\setminus f^{-1}(0)\iso \Xf\setminus\partial\Xf$. Moreover, we have $\partial\Xf\iso f^{-1}(0)\times S^1$.

We denote by $X^*_\infty$ the bordered space $(X^*,X)\iso(X^*,\Xf)$, and we fix the following notation for the morphisms (all of them inclusions except for $\varpi_f$) that will appear throughout the paper:
\begin{equation}\label{eq:morphXf}
\begin{tikzcd}\partial\Xf\arrow[hook, r, "\tilde i_f"] & \Xf\arrow[bend left=40,rr,"\varpi_f"]&X^*_\infty\arrow[l,"\fj",swap,hook']\arrow[r,hook,"j"] & X\end{tikzcd}
\end{equation}

This construction is functorial in the following sense: Given two complex manifolds $X$ and $Y$ with holomorphic functions $f\colon X\to \C$ and $g\colon Y\to \C$, as well as a morphism of complex manifolds $\tau \colon X\to Y$ such that $g\circ\tau=f$, then there is an induced morphism $\widetilde{\tau}\colon \Xf \to \widetilde{Y}_g$ such that the following diagram commutes:
$$\begin{tikzcd}
\Xf\arrow{dd}[swap]{\varpi_f}\arrow{rr}{\widetilde{\tau}}&& \widetilde{Y}_g\arrow{dd}{\varpi_g}\\ \\
X\arrow{rr}{\tau}&& Y
\end{tikzcd}$$

In particular, for $X=\C$ and $f(z)=z$, this construction gives the real blow-up space $\wt{\C}_0=\R_{\geq 0}\times S^1$ with $\varpi_0\colon \wt{\C}_0\to \C$ given by $\varpi_0(\rho,e^{i\theta})=\rho e^{i\theta}$.

For arbitrary $X$ and holomorphic $f\colon X\to \C$, one can then alternatively describe the construction of $\Xf$ as the fiber product $\Xf\defeq X\times_\C \wt{C}_0$, i.e., the space fitting into the Cartesian diagram
$$\begin{tikzcd}
\Xf \arrow{rr}{} \arrow{dd}[swap]{\varpi_f} \ar[phantom, ddrr, "\square"] && \wt{\C}_0 \arrow{dd}{\varpi_0}\\ \\
X\arrow{rr}[swap]{f} && \C
\end{tikzcd}$$

\subsubsection*{Real blow-up along a simple normal crossing divisor}
Let $X$ be a complex manifold and $D\subset X$ a normal crossing divisor with smooth components. (Such a normal crossing divisor is often said to be \emph{simple} or \emph{strict}, but we will always mean such a divisor even when we just say ``normal crossing divisor''.)
Locally, one can write $D=\{z_1\cdot\ldots\cdot z_r=0\}=D_1\cup\ldots\cup D_r$ for an appropriate coordinate system $z_1,\ldots,z_n$ and $r\leq n$, where we write $D_j \defeq\{z_j=0\}$.
Then, setting $f_j(z_1,\ldots,z_n)=z_j$, we can define the real blow-up spaces $\varpi_{f_j}\colon \widetilde{X}_{f_j}\to X$ as above. (This amounts to replacing the coordinates $z_1,\ldots,z_r$ by polar coordinates $\rho_1,e^{i\theta_1},\ldots,\rho_r,e^{i\theta_r}$, where we allow $\rho_j=0$.) The \emph{real blow-up space of $X$ along (the components of) $D$} is then denoted by $\varpi_D\colon \XD\to X$ and defined as the fiber product $\wt{X}_{f_1}\times_X\ldots\times_X\wt{X}_{f_r}$. This map is a homeomorphism on $X^*\vcentcolon=X\setminus D\iso \XD\setminus \partial \XD$.

We fix the following notation for the morphisms (all of them inclusions except for $\varpi_D$) that will appear throughout the paper:
\begin{equation}\label{eq:morphXD}
\begin{tikzcd}\partial \XD \arrow[hook, r, "\tilde i_D"] & \XD\arrow[bend left=35,rr,"\varpi_D"]&(X\setminus D)_\infty\arrow[l,"\;\Dj",swap,hook']\arrow[r,hook,"j"] & X\end{tikzcd}
\end{equation}

In the context of a normal crossing divisor, one can also define the \emph{total real blow-up} $\Xtot$ (which amounts to allowing $\rho_i\in\R$ in the construction above). It contains the real blow-up $\XD$ as a closed subset, but---contrarily to the latter---is a real analytic manifold, which will be useful in some situations. Note, however, that it is not globally intrinsically defined (see \cite[Remark 7.1.1]{DK16} and \cite[Remark 4.2.1]{KS16}). We will require $\Xtot$ later on in Subsection~\ref{subsec:SabbahProof}.

\subsubsection*{Comparison between the two constructions}
While the second definition requires the divisor to have normal crossings and smooth components, the first definition works in a more general setting. In the case of a normal crossing divisor given as the zero set of a holomorphic function, we therefore have two different notions of real blow-up spaces: Let $D$ be a normal crossing divisor given by $D=f^{-1}(0)$ for some holomorphic function $f\colon X\to \C$, then we can define the spaces $\XD$ and $\Xf$, and there is a natural proper morphism $\varpi_{D,f}\colon \XD\to \Xf$ and a commutative diagram
$$\begin{tikzcd}
\XD\arrow{rr}{\varpi_{D,f}}\arrow{dr}[swap]{\varpi_D}&&\Xf\arrow{ld}{\varpi_f}\\
&X
\end{tikzcd}$$

It is easy to see that in the example $X=\C$, with divisor $D= \{0\}$ with defining function $f(z)=z$, the two constructions of real blow-ups $\Xf$ and $\XD$ coincide. We explore this special case in more detail below in Subsection~\ref{subsec:dim1case}.

\subsubsection*{Sheaves of holomorphic functions with moderate growth and rapid decay}
Let $\varpi\colon \widetilde{X}\to X$ denote any of the real blow-up spaces defined above. Then one has the sheaves $\Amod$ and $\Ard$ of functions on $\widetilde{X}$ which are holomorphic in the interior $X^*$ and have moderate growth and rapid decay at the boundary $\partial\widetilde{X}$, respectively.

For details on these notions, see, e.g., \cite[§8.3]{Sab13}, \cite[§4.1]{Moc14}, \cite[§7.2]{KS01}, \cite[§5.1]{DK16}. Note in particular that $\Amod$ and $\Ard$ are flat over $\varpi^{-1}\cO_X$ (see, e.g., \cite[Theorem~4.1.1, Theorem~4.1.5]{Moc14}).

\begin{rem}
It is worth noting that the definitions of moderate (or polynomial) growth are phrased slightly differently in different works such as \cite{KS01}, \cite{Sab13} and \cite{Moc14}. However, they all lead to the same sheaf of holomorphic functions with moderate growth at the boundary.

Our notation $\Amod$ is closest to that in \cite{Sab13} and \cite{Moc14}. This sheaf is denoted by $\cA_{\widetilde{X}}$ in \cite{DK16}. however, it is not the same as the sheaf $\cA_{\widetilde{X}}$ in \cite{Sab00}, and also should not be confused with the sheaf $\mathscr{A}_{\widetilde{X}}$ in \cite{Sab13} or the sheaf $\cA$ in \cite{Mal91}.

In fact, it is not completely obvious from the beginning that the sheaf $\Amod$, as defined above, is the same as the sheaf $\cA_{\widetilde{X}}$ defined in \cite[Notation 7.2.1]{DK16}: In loc.~cit., it is defined as the sheaf of functions that are holomorphic in the interior of $\widetilde{X}$ and tempered at the boundary $\partial \widetilde{X}$. This is a priori a stronger condition than the one imposed above (and in \cite{Sab13}, \cite{Moc14}, for example), since \emph{tempered} means that the function \emph{and all its derivatives} are of moderate growth. For holomorphic functions with moderate growth, this is, however, automatic due to Cauchy's integral formula for derivatives, as is shown in \cite[Lemma 3]{Siu70}.
\end{rem}

Since differentiation of holomorphic functions preserves the moderate growth condition, the sheaf $\varpi^{-1}\cD_X$ naturally acts on $\Amod$. This allows us to define a sheaf of differential operators on real blow-up spaces, by setting
\begin{align*}
    \cD_{\wt X}^\cA \vcentcolon = \varpi^{-1}\cD_X \otimes_{\varpi^{-1}\cO_X} \Amod.
\end{align*}
We similarly obtain an analogue of $\Omega_X$ in the right $\cD_{\wt X}^\cA$-module 
$$\Omega_{\wt X} \vcentcolon= \varpi^{-1}\Omega_X \otimes_{\varpi^{-1}\cO_X} \Amod.$$
\subsubsection*{De Rham complexes on real blow-ups}
Let $X$ be a complex manifold and let a normal crossing divisor $D\subset X$ or a holomorphic function $f\colon X\to \C$ be given, and denote by $\wt{X}$ any of the real blow-ups $\XD$ or $\Xf$.

\begin{defi}\label{def:modDRsabbah}
The \emph{moderate growth de Rham complex} of a holonomic $\cD_X$-module $\cM$ is defined as (see, e.g., \cite{Sab13}\footnote{We remark that we follow here the convention of \cite{Kas03}: The objects $\DR$, $\DRmod$ etc.\ here are those denoted by ${}^p\DR$, ${}^p\DRmod$ in \cite{Sab13}, i.e., these objects are already shifted appropriately such that, for instance, $\DR(\cM)$ is a perverse sheaf.})
\begin{align*}
    \DRmod(\cM)&\defeq (\varpi^{-1}\Omega_X\otimes_{\varpi^{-1}\cO_X} \Amod) \Ltens_{\varpi^{-1}\cD_X}\varpi^{-1}\cM
    \\
    &\iso \varpi^{-1}\Omega_X\Ltens_{\varpi^{-1}\cD_X} (\Amod\otimes_{\varpi^{-1}\cO_X}\varpi^{-1}\cM).
\end{align*}
We set $\cM^\cA\vcentcolon=\Amod\otimes_{\varpi^{-1}\cO_X}\varpi^{-1}\cM \iso \cD_{\wt X}^\cA \otimes_{\varpi^{-1}\cD_X} \varpi^{-1}\cM$, so that $\cM^\cA$ is a left $\cD_{\wt X}^\cA$-module. 

Similarly, the \emph{rapid decay de Rham complex} is defined as
\begin{align*}
    \DRrd(\cM)&\defeq (\varpi^{-1}\Omega_X\otimes_{\varpi^{-1}\cO_X} \Ard) \Ltens_{\varpi^{-1}\cD_X}\varpi^{-1}\cM
    \\
    &\iso \varpi^{-1}\Omega_X\Ltens_{\varpi^{-1}\cD_X} (\Ard\otimes_{\varpi^{-1}\cO_X}\varpi^{-1}\cM).
\end{align*}

Let $\cN \in \mathrm{D}^\mathrm{b}(\cD_{\wt X}^\cA)$. We define the de Rham complex of $\cN$ to be the object 
$$
\mathrm{DR}_{\wt X}(\cN) \vcentcolon= \Omega_{\wt X} \Ltens_{\cD_{\wt X}^\cA} \cN.
$$

\end{defi}

\section{Solutions to the Riemann--Hilbert problem for holonomic $\cD$-modules}\label{sec:motivation}

In this section, we will give some motivation on the constructions that are used in this work, in particular in the context of the irregular Riemann--Hilbert correspondence of D'Agnolo--Kashiwara. We do this by explaining mainly the case of an exponential $\cD$-module on a complex curve. All of the results here are well-known, but might help the reader to understand better the background of the theory used later in this article.

Let $X$ be a complex manifold and $D \subset X$ a divisor, and let $g \in \cO_X(*D)$ be a meromorphic function with poles contained in $D$. Then, $\cE_{X\setminus D | X}^g$ is a natural example of a holonomic $\cD_X$-module with irregular singularities along $D$, called an \emph{exponential $\cD_X$-module} (with exponent $g$). This is a rank one $\cO_X(*D)$-module with connection $d-dg$, and its image under the classical de Rham functor is the shifted constant sheaf
$$
\DR(\cE_{X\setminus D | X}^g) \iso \C_{X \setminus D}[\dim X].
$$
That is, the classical de Rham functor ``forgets" the exponent, and thus identifies all exponential $\cD_X$-modules with poles along $D$.
However, the enhanced de Rham functor fixes this issue:
\begin{align*}
\DRE(\cE_{X\setminus D | X}^g) \iso \RIhom(\C_{X\setminus D},\mathbb{E}_{X \setminus D|X}^{\mathrm{Re}(g)})[\dim X] \end{align*}
(see Lemma 9.3.1 of \cite{DK16}).
The type of enhanced ind-sheaf thus obtained, called an enhanced exponential ind-sheaf, is of fundamental importance to understanding the irregular Riemann--Hilbert correspondence. We elaborate more below.

Let $U \subseteq X$ be an open subset, and let $h \colon U \to \R$ be continuous. To this data, for any $c \in \R$ we associate the sheaf $\C_{\{t+h\geq c\}}$ on $X \times \R$, where
$$
\{t+h \geq c\} \vcentcolon= \{(x,t) \in X \times \R \, | \, x \in U, t+h(x) \geq c\}.
$$
and take the ind-limit over $c \to +\infty$, considered as an object in $\EbIC{X}$:
\begin{align*}
\mathbb{E}_{U|X}^h &\vcentcolon= \indlim{c \to +\infty} \C_{\{t+h\geq c\}} \\
&\iso \C_X^\EE \conv \C_{\{t+h\geq 0\}}
\end{align*}
We call such an object an \emph{exponential enhanced ind-sheaf} (with exponent $h$), analogously to the $\cD$-module setting above. Loosely, one cares about such objects for their ability to keep track of the asymptotic behavior of the exponent $h$. 

\medskip

\subsection{The dimension one case}\label{subsec:dim1case}
Revisiting the above example in the case where $X = \C$, $D = \{z= 0\}$ and $g \in \cO_X(*0)$, we take $U = X^* \vcentcolon= X \setminus \{0\}$. Throughout this section, we also discuss only the \emph{unramified} case. Then, the de Rham complex of the exponential $\cD_X$-module $\mathcal{E}_{X^*|X}^{g}$ is isomorphic to $\C_{X^*}[1]$; the stalk of this local system at a non-zero point is spanned by the function $e^{g(z)}$. 

Whether or not this exponential function has moderate growth at $0$ is determined by certain bounds on the growth of $|e^{g(z)}| = e^{\mathrm{Re}(g(z))}$ as $z \to 0$.

Let $j\colon X^* \hookrightarrow X$ be the open embedding, $W \subseteq X$ be open, and let $u$ be a section of $j_*j^{-1}\cO_X$ on $W$, in particular $u(z)$ determines a $\C$-valued holomorphic function on $W\setminus\{0\}$. We say that $u$ has \emph{moderate growth at $0$} if there exists a sufficiently small neighborhood $U$ of $0$ such that there exist constants $C >0$ and $N \in \mathbb{N}$ such that the bound
$$
|u(z)|\leq C|z|^{-N}
$$
holds for all $z \in (U\cap W)\setminus \{0\}$.

We say $u$ has \emph{moderate growth at angle $\theta \in S^1$} if $u(z)$ has moderate growth on some angular neighborhood $W \subseteq X^*$ of $\theta$. That is, there exists $r>0$ and $0<\epsilon \ll 1$ such that $u(z)$ has moderate growth on the subset $\{z \in X \, | \, 0<|z|<r, \mathrm{arg}(z) \in (\theta-\epsilon,\theta+\epsilon)\}$.

Generally, this bound is not satisfied for exponential functions of the form $u(z)=e^{g(z)}$ at every angle $\theta \in S^1$ (e.g., this occurs if and only if the exponent $g$ is bounded near $0$). Instead, if $g$ has a pole at $0$, $e^{g(z)}$ has moderate growth only at those angles $\theta$ having an angular neighborhood on which $\mathrm{Re}(g(z)) < 0$. To separate out such angles (which we can without loss of generality picture as half-lines originating at $0$), we pass to the space of polar coordinates of $X$ at $0$, called the \emph{real blow-up of $X$ at $0$}:
\begin{align*}
\wt X_0 \vcentcolon= \R_{\geq 0} \times S^1 &\xlongrightarrow{\varpi} \C \\
(r,e^{i\theta}) &\mapsto re^{i\theta}
\end{align*}

Extending $\mathcal{E}_{X^*|X}^g = (\cO_X(*0),d-dg)$ to a $\cD$-module on $\wt X_0$ requires a bit of care, as $\R_{\geq 0} \times S^1$ is not a complex manifold (it is a smooth manifold with boundary), so we must define an analogue of the sheaf $\cO_X$ for $\wt X_0$. We recall the morphisms of Subsection \ref{subsec:realblowup}: 
\begin{equation}
\begin{tikzcd}\partial \wt X_0 \arrow[hook, r, "\tilde i"] & \wt X_0 \arrow[bend left=35,rr,"\varpi"]&X^*\arrow[l,"\;\tilde j",swap,hook']\arrow[r,hook,"j"] & X\end{tikzcd}
\end{equation}

 One approach is to use a certain subsheaf of $\tilde j_*\cO_{X^*} \iso \tilde j_* j^{-1}\cO_X$ of smooth functions on $\wt X_0$ that are holomorphic on the interior $\wt X_0 \backslash \partial \wt X_0 = X^*$ and have \emph{moderate growth} along the boundary $\partial \wt X_0 \vcentcolon= \varpi^{-1}(0) \iso S^1$, denoted $\cA_{\wt X_0}^{\mathrm{mod}}$ (for the exact definition of this sheaf, see, e.g., Subsection 8.3 of \cite{Sab13} or Section 4 of \cite{Moc14}). Using the sheaf of moderate growth functions, we define the \emph{moderate growth de Rham complex} of $\cM \in \DbholD{X}$ as an object of $\DbC{\widetilde{X}_0}$ which, loosely, is the complex whose stalk at a point $z \in \wt X_0$ consists of local flat sections of $\cM$ that have moderate growth at the point $z$. Precisely, we set
\begin{align*}
\mathrm{DR}_{\wt X_0}^{\mathrm{mod}}(\cM) \vcentcolon= \varpi^{-1}\Omega_X\Ltens_{\varpi^{-1}\cD_X} (\cA_{\wt X_0}^{\mathrm{mod}} \otimes_{\varpi^{-1}\cO_X}\varpi^{-1}\cM)
\end{align*}

This is easy to compute for our exponential $\cD_X$-module above: if $g$ has an effective pole at $0$, we have
\begin{align*}
\mathrm{DR}_{\wt X_0}^{\mathrm{mod}}(\mathcal{E}_{X^*|X}^g) \iso \C_{I_g \sqcup X^*} [1] 
\end{align*}
where 
$$
I_g \vcentcolon= \partial \wt X_0 \setminus \overline{ \{z \in X^* \, | \, \mathrm{Re}(g(z))\geq 0\}}
$$
with the closure taken inside $\wt X_0$. This interval $I_g \subset S^1$ consists of those angles $\theta$ having an open neighborhood in $\wt X_0$ on which $e^g$ has moderate growth. If $g \in \cO_X(*0)/\cO_X \iso z^{-1}\C\{z^{-1}\}$ has pole order $m>0$, then $I_g$ is a disjoint union of $m$ open (equally spaced) intervals on $S^1$ (cf.\ example 3.11 of \cite{DK21sh}). When $g=0$ (or, more generally, $g \in \cO_X$), we set $I_0 \vcentcolon= S^1$.

\medskip

\subsection{Stokes filtrations} Noticing that the complex $\mathrm{DR}_{\wt X _0}^{\mathrm{mod}}(\mathcal{E}_{X^*|X}^g)$ only tells us ``new" information about $\mathcal{E}_{X^*|X}^g$ along the boundary of $\wt X_0$, it is often useful to just consider its restriction to this subset. For general $\cM \in \DbholD{X}$, we define the object
$$
\Psi_0^{\leq 0}(\cM) \vcentcolon= \tilde i^{-1}\mathrm{DR}_{\wt X _0}^{\mathrm{mod}}(\cM)[-1],
$$
called the \emph{moderate growth nearby cycles} of $\cM$ at $0$. On exponentials, this is easy to compute:
\begin{align*}
\Psi_0^{\leq 0}(\mathcal{E}_{X^*|X}^g) \iso \C_{I_g}.
\end{align*}

 The superscript ``$\leq 0$" in the functor $\Psi_0^{\leq 0}$ refers to a filtration by $\R$-constructible objects of $\DbC{S^1}$ called the \emph{Stokes filtration}. The failure of the classical de Rham functor to identify the exponent $g$ in the isomorphism $\DR(\mathcal{E}_{X^*|X}^g) \iso \C_{X^*}[1]$ was originally remedied with the use of the Stokes filtration (cf., e.g., \cite{Mal91}), and we will quickly sketch this argument (following \cite{Sab13}, see also \cite{Sab19}). 

Denote by $\cL_g \vcentcolon= \tilde i^{-1}\tilde j_* j^{-1}\DR(\mathcal{E}_{X^*|X}^g)[-1] \iso \C_{S^1}$ the constant sheaf on $S^1$ with stalk $\C$, and for any meromorphic function $h$ on $X$ with poles contained in $0$, define the following $\R$-constructible subsheaf of $\cL_g$:
\begin{align*}
\mathrm{F}_{\leq h}\cL_g \vcentcolon= \Psi_0^{\leq h}(\mathcal{E}_{X^*|X}^g) &\vcentcolon = \Psi_0^{\leq 0}(\mathcal{E}_{X^*|X}^g \overset{D}{\otimes} \mathcal{E}_{X^*|X}^{-h}) \\
&\iso \Psi_0^{\leq 0}(\mathcal{E}_{X^*|X}^{g-h}) \\
&\iso \C_{I_{g-h}}.
\end{align*}

Notice that whenever the function $e^{h-h'}$ has moderate growth at some angle $\theta$, there is a monomorphism 
$$
\mathrm{F}_{\leq h}\cL_g \hookrightarrow \mathrm{F}_{\leq h'}\cL_g
$$
on some open neighborhood of $\theta$ on $S^1$. The collection of subsheaves $\mathrm{F}_{\leq h}\cL_g$ is a filtration of $\cL_g$ indexed by meromorphic functions. (Recall that we consider the non-ramified case here for simplicity; in general, one considers multi-valued meromorphic functions $h$ and considers only open angular sectors about the origin on which the Puiseux expansion of $h$ converges. To simplify the exposition we do not do this, see, e.g., \cite{Sab13} for a more complete treatment), and we claim the data of this filtration allows us to recover our original exponent $g$ (up to possibly adding a holomorphic function). Suppose that, for some other exponent $g'$, the local systems $\mathrm{F}_{\leq h}\cL_g$ and $\mathrm{F}_{\leq h}\cL_{g'}$ are equal for all $h$. Consequently, it must be the case that for $h=g$,
\begin{align*}
\mathrm{F}_{\leq g}\cL_{g'} = \mathrm{F}_{\leq g}\cL_g \iso \Psi_0^{\leq 0}((\cO_X,d)) \iso \C_{S^1}
\end{align*}
i.e., $e^{g-g'}$ has moderate growth at all angles $\theta \in S^1$, which can only happen if $g-g' \in \cO_X$ is bounded in a neighborhood of $0$ as mentioned above. Therefore the Stokes filtration $\{\mathrm{F}_{\leq h}\cL_g\}_h$ on the local system $\cL_g$ is uniquely determined by $g \in \cO_X(*0)/\cO_X$. Since $\cE_{X^*|X}^g \iso \cE_{X^*|X}^{g+g_0}$ for any $g_0 \in \cO_X$, we are done. 

\medskip

In general, this procedure can be extended to give an equivalence between the category of meromorphic connections on $X$ with poles contained in $\{0\}$ and the category of Stokes-filtered local systems on $\wt X_0$:
\begin{align*}
\cM &\mapsto (\cL_\cM, \mathrm{F}_{\leq h}\cL_\cM)
\end{align*}
where $\cL_\cM \vcentcolon= H^0(\tilde i^{-1}R \tilde j_*j^{-1}\DR(\cM)[-1])$, and 
\begin{align*}
\mathrm{F}_{\leq h}\cL_\cM \vcentcolon= \Psi_0^{\leq 0}(\cM \overset{\mathrm{D}}{\otimes} \mathcal{E}_{X^*|X}^{-h}). 
\end{align*}

The advantage of this approach to the irregular Riemann--Hilbert correspondence is that it produces topological objects (Stokes-filtered local systems) that are very concrete---they are collections of classical $\R$-constructible sheaves. The issue is that it is very difficult to generalize to higher dimensional manifolds $X$ and meromorphic connections with poles contained in a normal crossing divisor $D \subset X$, let alone arbitrary objects of the derived category $\DbholD{X}$. Moreover, it is difficult to understand the behavior of the Stokes filtration under the usual type of six functor formalism we would like. This is in stark contrast with the irregular Riemann--Hilbert correspondence of D'Agnolo--Kashiwara, which works in all dimensions and for all of $\DbholD{X}$ and has a robust functorial framework, but the objects of $\EbIC{X}$ are themselves much harder to understand. Stokes phenomena in higher dimensions have been previously studied by various people (see, for example, \cite{Maj84} and \cite{Sab00}, and \cite{Moc11b} for a survey) from the $\cD$-module perspective and recently in \cite{Moc22a,Moc22b} with enhanced ind-sheaves.

As a consequence, it is then natural to want to further develop the ``dictionary" that allows us to move between these versions of the Riemann--Hilbert correspondence; to express the moderate growth de Rham complex of $\cM$ entirely in terms of the enhanced ind-sheaf $\DRE(\cM)$, and similarly express the Stokes-filtered local system associated to $\cM$ in terms of $\DRE(\cM)$ in the case where $\cM$ is a meromorphic connection on $X$ with poles contained in $D$. This will give better intuition for both correspondences, and we believe lend a greater understanding of the topological data contained in the relatively recently-defined category of (constructible) enhanced ind-sheaves. 

In dimension one, this was already accomplished by D'Agnolo--Kashiwara \cite{DK20cycles}; given a holonomic $\cD_X$-module $\cM$ on $X=\C$ with poles contained in the divisor $D= \{0\}$, 
\begin{align}\label{eq:DKleq0}
\mathrm{DR}_{\wt X_0}^{\mathrm{mod}}(\cM) &\iso \mathrm{sh}_{\wt X_0} (\EE \tilde j_* \EE j^! \DRE(\cM)) \\
\Psi_0^{\leq 0}(\cM) &\iso \tilde i^{-1}\mathrm{sh}_{\wt X_0}(\EE \tilde j_* \EE j^! \DRE(\cM))[-1]\notag
\end{align}
where the graded piece of the Stokes filtration on $\cM$ in degree 0 is given by
$$
\Psi_0^0(\cM) \iso \mathrm{sh}_{\partial \wt X_0}(\EE \tilde i^{-1}\EE \tilde j_* \EE j^!\DRE(\cM))[-1].
$$
Likewise, for a meromorphic function $h$ at $0$, the enhanced-ind sheaf analogue of the exponentially twisted $\cD$-module $\cM(h) \vcentcolon= \cM \overset{D}{\otimes} \mathcal{E}_{X^*|X}^{-h}$ is given by 
\begin{align*}
\DRE(\cM(h)) &\iso \RIhom^+(\mathbb{E}_{X^*|X}^{\mathrm{Re}(h)},\DRE(\cM))
\end{align*}
Therefore, 
\begin{align*}
\Psi_0^{\leq h}(\cM) &\iso \Psi_0^{\leq 0}(\cM(h)).
\end{align*}
D'Agnolo--Kashiwara also obtain and study an analogue of $\cA_{\wt X_0}^{\mathrm{mod}}$ on $\wt X_0$ using ind-sheaves and the theory of tempered distributions, called the ind-sheaf of \emph{tempered holomorphic functions} $\cO_{\wt X_0}^{t}$ on $\wt X_0$. This object satisfies the relationships $\alpha \Ot{\wt X_0} \iso \cA_{\wt X_0}^{\mathrm{mod}}$ in $\mathrm{D}^\mathrm{b}(\varpi^{-1}\cD_X)$, and $\Ot{\wt X_0} \iso \varpi^!\cO_X^t(*D)$ in $\mathrm{D}^\mathrm{b}(\mathrm{I}\varpi^{-1}\cD_X)$. This approach works in all dimensions and for any normal crossing divisor $D$ on the real blow-up $\XD$, but some work must be done in the case of the blow-up along arbitrary functions $f$ on $X$ (see Lemma-Definition \ref{lemmadefi:OtXf}).

\medskip

\subsection{Duality observations on the real blow-up}
With exponential functions, we immediately notice an interesting relationship between the angles of moderate growth of $e^g$ and that of $e^{-g}$ when $g$ has an effective pole at $0$, arising as flat sections of the $\cD$-module $\cE_{X^*|X}^g$ and its dual $\cE_{X^*|X}^{-g} \iso (\mathbb{D}_X \cE_{X^*|X}^g)(*0)$. That is, if $e^g$ has moderate growth at angle $\theta \in S^1$ (so that $\mathrm{Re}(g(z)) <0$ on an angular neighborhood of $\theta$), then $e^{-g}$ will \emph{not} have moderate growth at $\theta$; additionally, if $\theta$ is not an angle corresponding to a half tangent to the curve $\mathrm{Re}(g(z))=0$ at $0$ (called the Stokes line of $e^g$), then exactly one of $e^g$ and $e^{-g}$ will have moderate growth at $\theta$. We wish to explore this observation by comparing their moderate growth nearby cycles at $0$.  By our previous computation, these sheaves are isomorphic to $\C_{I_g}$ and $\C_{I_{-g}}$, respectively. 

\smallskip

Since the $\cD_X$-modules are exchanged by duality $\cE_{X^*|X}^g \iso (\mathbb{D}_X\cE_{X^*|X}^{-g})(*0)$, it is natural to ask if the corresponding nearby cycles are similarly exchanged by \emph{Verdier} duality on $\DbC{S^1}$; however, this is not the case: 
$$
\Psi_0^{\leq 0}(\cE_{X^*|X}^g) \iso \C_{I_g} \ncong \mathrm{D}_{S^1}(\C_{I_{-g}}) \iso \mathrm{D}_{S^1}\Psi_0^{\leq 0}(\cE_{X^*|X}^{-g}).
$$
Instead, the correct duality statement for arbitrary holonomic objects $\cM \in \DbholD{X}$ is 
\begin{align}\label{eqn:duality0}
\Psi_0^{\leq 0}(\mathbb{D}_X\cM) \iso \mathrm{D}_{S^1}\Psi_0^{\geq 0}(\cM)[1],
\end{align}
where, loosely, $\Psi_0^{\geq 0}$ records local flat sections of $\cM$ that have \emph{greater than moderate growth} along various angles of $S^1$. Like $\Psi_0^{\leq0}(\cM)$, $\Psi_0^{\geq 0}(\cM)$ can be obtained from the restriction of a certain de Rham complex with growth conditions on $\wt X_0$, called the \emph{rapid decay de Rham complex}, denoted $\mathrm{DR}_{\wt X_0}^{\mathrm{rd}}(\cM)$. We then have a natural isomorphism
$$
\Psi_0^{\geq 0}(\cM) \iso {\tilde i}^!\mathrm{DR}_{\wt X_0}^{\mathrm{rd}}(\cM).
$$

\medskip

\subsection{Rapid decay de Rham complexes}\label{subsec:rapidDR}
Similar to moderate growth, the \emph{rapid decay condition} is a growth condition on functions. Precisely, a section $u \in j_* j^{-1} \cO_X$ on an open set $W \subseteq X$ has \emph{rapid decay at $0$} if there exists a neighborhood $U\subseteq X$ of $0$ such that for every $N \in \mathbb{N}$, there exists a constant $C_N$ such that the estimate
$$
|u(z)| \leq C_{N}|z|^N
$$
holds for all $z \in (U \cap W)\setminus\{0\}$. As with moderate growth, we say $u$ has rapid decay at angle $\theta \in S^1$ if it has rapid decay on some angular neighborhood of $\theta$ in $X^*$.
The sheaf of holomorphic functions on $\wt X_0$ having rapid decay along $\partial \wt X_0$ is denoted $\cA_{\wt X_0}^{\mathrm{rd}}$, and we remark there are natural inclusions of sheaves on $\wt X_0$:
$$
\cA_{\wt X_0}^{\mathrm{rd}} \subset \cA_{\wt X_0}^{\mathrm{mod}} \subset \tilde j_*j^{-1}\cO_X.
$$
The \emph{rapid decay de Rham complex} of a holonomic $\cD_X$-module $\cM$:
\begin{align}\label{eqn:rapiddecayDR}
\mathrm{DR}_{\wt X_0}^{\mathrm{rd}}(\cM) = \varpi^{-1}\Omega_X \Ltens_{\varpi^{-1}\cD_X} (\mathcal{A}_{\wt X_0}^{\mathrm{rd}} \otimes_{\varpi^{-1}\cO_X} \varpi^{-1}\cM)
\end{align}
and the \emph{rapid decay nearby cycles} of $\cM$ at $0$ is analogously defined as the restriction
\begin{align}\label{eqn:rapiddecaynearby}
\Psi_0^{<0}(\cM) \vcentcolon= \tilde i^{-1}\mathrm{DR}_{\wt X_0}^{\mathrm{rd}}(\cM)[-1].
\end{align}
The superscript ``$<0$" is again related to the Stokes filtration. Consider $\cL_\cM = \tilde i^{-1}\tilde j_* j^{-1}\DR(\cM)[-1]$, and define $\mathrm{F}_{<0}\cL_\cM \vcentcolon= \Psi_0^{<0}(\cM)$. The Stokes filtration yields the following short exact sequence of sheaves, for any meromorphic connection $\cM$ on $X$ with poles contained in $0$:
$$
0 \to \mathrm{F}_{<0}\cL_\cM \to \cL_\cM \to \Psi_0^{\geq 0}(\cM) \to 0
$$
where we sometimes write $\mathrm{F}^{\geq 0}\cL_\cM \vcentcolon= \Psi_0^{\geq 0}(\cM)$. Likewise, for such connections there is an inclusion of sheaves $\mathrm{F}_{<0}\cL_\cM \subseteq \mathrm{F}_{\leq 0}\cL_\cM$ which fit into the short exact sequence 
\begin{align*}
0 \to \mathrm{F}_{<0}\cL_\cM \to \mathrm{F}_{\leq 0}\cL_\cM \to \mathrm{Gr}_0^{\mathrm{F}} \cL_\cM \to 0.
\end{align*}
Here, $\mathrm{Gr}_0^{\mathrm{F}}\cL_\cM \iso \Psi_0^0(\cM)$ is the graded piece of the Stokes filtration on $\cL_\cM$ in degree $0$. We will further study this topic later in Subsection \ref{subsec:rapiddecay} and Section \ref{sec:enhmodrd} from the perspective of enhanced ind-sheaves.

\medskip

\subsection{Duality revisited}\label{subsec:dualdim1}

With this in hand, the duality isomorphism \eqref{eqn:duality0} stated above can be rephrased as a perfect pairing
\begin{align}\label{eqn:duality1}
\tilde i^{-1}\mathrm{DR}_{\wt X_0}^{\mathrm{mod}}(\cM)[-1] \otimes \tilde i^! \mathrm{DR}_{\wt X_0}^{\mathrm{rd}}(\mathbb{D}_X\cM) \longrightarrow \omega_{S^1}
\end{align}
in $\DbC{S^1}$, where $\omega_{S^1}\iso \C_{S^1}[1]$ is the Verdier dualizing complex on $S^1$. This suggests the more general duality for $\cM \in \DbholD{X}$, which we write again as a pairing:
\begin{align}\label{eqn:duality2}
\mathrm{DR}_{\wt X_0}^{\mathrm{mod}}(\cM) \otimes \mathrm{DR}_{\wt X_0}^{\mathrm{rd}}(\mathbb{D}_X \cM) \longrightarrow \omega_{\wt X_0}.
\end{align}
This duality was first proven in the curve case by S.\ Bloch and H.\ Esnault \cite{BE04}, where it was phrased cohomologically as a perfect pairing between the algebraic de Rham cohomology of a flat algebraic connection $(E,\nabla)$ on $X^*$ and the rapid decay homology of the dual connection $(E^\vee,\nabla^\vee)$ on a good compactification $X$ of $X^*$. The higher-dimensional case was proven by M.\ Hien \cite{Hie09} for arbitrary flat algebraic connections $(E,\nabla)$ on smooth quasi-projective algebraic varieties $U$ over $\C$, where one must carefully choose a \emph{good compactification} of $U$ with respect to $(E,\nabla)$---i.e., $X \backslash U = D$ is a normal crossing divisor along which the flat meromorphic connection $(E(*D),\nabla)$ has good formal structure. Then, Hien's pairing is described both at the level of sheaves as in (\ref{eqn:duality2}), and on global sections as a perfect pairing of finite-dimension $\C$-vector spaces
\begin{align}\label{eqn:duality3}
\mathrm{H}_{\mathrm{dR}}^m(U;(E,\nabla)) \otimes \mathrm{H}_m^{\mathrm{rd}}(U;(E^\vee,\nabla^\vee)) \longrightarrow \C,
\end{align}
where $\mathrm{H}_{\mathrm{dR}}$ denotes algebraic de Rham cohomology, and $\mathrm{H}_m^{\mathrm{rd}}$ denotes rapid decay homology. We will revisit the pairing (\ref{eqn:duality3}) in Subsection \ref{subsec:dualityBEH}.

\section{De Rham complexes with growth conditions and enhanced ind-sheaves}\label{sec:DRmodandrd}

In this section, we recall and put together some notions and constructions from previous works on enhanced ind-sheaves. The main aim is to motivate the formulas for our definitions in the following sections, by making the connection to the moderate growth and rapid decay de Rham functors. 

\subsection{Moderate growth de Rham complexes in higher dimensions: normal crossing divisors}
We will now show how the ideas of \cite{DK20cycles} naturally extend to the higher-dimensional case for $\cD_X$-modules on a complex manifold $X$ with a (simple) normal crossing divisor $D$.

Consider the real oriented blow-up $\XD$ along the components of $D$ and recall the notation for morphisms from \eqref{eq:morphXD}.

As the next proposition shows, a formula similar to \eqref{eq:DKleq0} recovers the moderate growth de Rham complex from the enhanced de Rham complex of a holonomic $\cD_X$-module in the normal crossing case. Its proof is probably not new to experts and relies heavily on results already established in \cite{DK16}, but we give it here since it does not seem to appear in other works. In particular, it gives a direct proof for the fact that the object \eqref{eq:DKleq0} recovers the moderate growth part of the Stokes filtration in the one-dimensional case.

\begin{prop}\label{prop:DRmod}
Let $\cM\in\DbholD{X}$ and $D\subset X$ a simple normal crossing divisor. Then there is an isomorphism in $\DbC{\XD}$
$$\shbuD\big(\EE \Dj_* \EE j^{-1} \DRE(\cM)\big)\iso \DRmodD(\cM).$$
 In particular, if $\mathrm{dim}_\C X=1$, we have
$$\Psi_a^{\leq 0}\big(\DRE(\cM)\big)\iso \mathrm{F}_{\leq 0}\cL_\cM,$$
where the right-hand side denotes the ``$\leq 0$'' part of the Stokes filtration on the local system $\cL_\cM\vcentcolon= i_D^{-1} \RR \Dj_* j^{-1} \DR(\cM)$.
\end{prop}

In the proof of this proposition, we will make use of the following lemma. It is the analogue for the de Rham functor on the real blow-up of \cite[Lemma 9.7.1]{DK16} (which shows that $\sh{X} \DRE(\cM)\iso\DR(\cM)$, although the sheafification functor was not denoted like this in \emph{loc.\ cit.}). We will use the constructions and notation for objects on the real blow-up along a normal crossing divisor from \cite[§7, §9.2]{DK16}.

\begin{lemma}\label{lemmaShDR}
Let $\cN\in\mathrm{D}^\mathrm{b}(\cD^\cA_{\XD})$. There is an isomorphism in $\DbC{\XD}$
$$\shbuD\DREbuD(\cN)\simeq \DRbuD(\cN).$$
\end{lemma}
\begin{proof}
	First, one has isomorphisms
	\begin{align*}
		\shbuD\DREbuD(\cN)&\iso \RHomE\big(\C^\EE_{\XD},\DREbuD(\cN) \big)\\
		&\iso \RHomE\big(\C^\EE_{\XD},\Omega^\EE_{\XD}\overset{\mathrm{L}}{\otimes}_{\cD^\cA_{\XD}} \cN  \big)\\
		&\iso\RHomE( \C^\EE_{\XD},\Omega^\EE_{\XD} ) \overset{\mathrm{L}}{\otimes}_{\cD^\cA_{\XD}} \cN,
	\end{align*}
where the last isomorphism follows with \cite[Theorem 5.6.1(ii)]{KS01}.
Therefore, it suffices to show that
$$\RHomE( \C^\EE_{\XD},\Omega^\EE_{\XD} )\iso \varpi_D^{-1}\Omega_{X}\otimes_{\varpi_D^{-1}\cO_X}\AmodD.$$

For this, we observe that
\begin{align*}
	\RHomE(\C_{\XD}^\EE,\cO_{\XD}^\EE)&\iso \RHomE\big(\C_{\XD}^\EE, \RIhom(\wt{\varpi}_D^{-1}\pi^{-1}\C_{X\setminus D},\EE\varpi_D^!\cO^\EE_X)\big)\\
	&\iso \RHomE\big(\wt{\varpi}_D^{-1}\pi^{-1}\C_{X\setminus D}\otimes\C_{\XD}^\EE,\EE\varpi_D^!\cO^\EE_X\big)\\
	&\iso \RHomE\big(\EE\varpi_D^{-1}(\pi^{-1}\C_{X\setminus D}\otimes\C^\EE_X),\EE\varpi_D^!\cO^\EE_X\big)\\
	&\iso \RHomE\big(\EE\varpi_D^{-1}\SolE(\cO_X(*D)),\EE\varpi_D^!\cO^\EE_X\big)\\
	&\iso \RHomE\big(\SolEbuD(\cO_X(*D)^\cA),\cO^\EE_{\XD}\big)\\
	&\iso \cO_X(*D)^\cA\iso \AmodD.
\end{align*}

Here, the first isomorphism follows from \cite[(9.6.7)]{DK16}, the second-to-last line follows from the computations in \cite[p.\ 192]{DK16} and the last line follows from \cite[(9.6.8)]{DK16}.

Consequently, we can conclude
\begin{align*}
	\RHomE\big(\C_{\XD}^\EE,\Omega_{\XD}^\EE \big)&\iso \RHomE\big(\C_{\XD}^\EE, \pi^{-1}\varpi_D^{-1}\Omega_{X}\otimes_{\pi^{-1}\varpi_D^{-1}\cO_{X}}\cO_{\XD}^\EE\big)\\
	&\iso \varpi_D^{-1}\Omega_X\otimes_{\varpi_D^{-1}\cO_X}\RHomE\big(\C_{\XD}^\EE,\cO_{\XD}^\EE\big)\\
	&\iso \varpi_D^{-1}\Omega_X\otimes_{\varpi_D^{-1}\cO_X}\AmodD,
\end{align*}
where the second isomorphism is due to \cite[Lemma 4.10.3]{DK16}.
\end{proof}

\begin{proof}[Proof of Proposition \ref{prop:DRmod}]
For the left-hand side, one has
\begin{align*}
 \shbuD\big(\EE \Dj_*\EE j^{-1}\DRE(\cM)\big) &\iso \shbu\big(\EE \Dj_*\EE j^!\DRE(\cM)\big)\\
    &\iso \shbuD\big(\EE {\Dj}_*\EE \tilde{j}_D^!\EE \varpi_D^! \DRE(\cM)\big)\\
	&\iso \shbuD\RIhom\big(\wt{\varpi_D}^{-1}\pi^{-1}\C_{X\setminus D},\EE \varpi_D^! \DRE(\cM)\big)\\
	&\iso \shbuD\EE\varpi_D^!\RIhom\big(\pi^{-1}\C_{X\setminus D},\DRE(\cM)\big)\\
	&\iso \shbuD \DREbuD(\cM^\cA).
\end{align*}
In the third isomorphism, we have used \cite[Lemma 2.7.6]{DK19} and the last isomorphism follows from \cite[§9]{DK16} (cf.\ also \cite[p.\ 13]{IT20}).

Since, by definition, we have $\DRmodD(\cM)=\DRbuD(\cM^\cA)$, we can now conclude with Lemma~\ref{lemmaShDR}.
\end{proof}

One can deduce a similar result for the moderate growth de Rham functor on the real blow-up $\Xf$ along the function $f$ defining a normal crossing divisor.

\begin{cor}\label{cor:differentblowups}
	Let $f\colon X\to \C$ be a holomorphic function such that $D=f^{-1}(0)\subset X$ is a simple normal crossing divisor. Let $\cM\in\DbholD{X}$. Then there is an isomorphism in $\DbC{\Xf}$
	$$\shbuf\big(\EE \fj_*\EE j^{-1} \DRE(\cM)\big)\iso \DRmodf(\cM).$$
\end{cor}
\begin{proof}
	There is a natural morphism $\varpi_{D,f} \colon \XD\to \Xf$ and one has $\fj=\varpi_{D,f} \circ \Dj$. Consequently,
	\begin{align*}
		\shbuf\big(\EE \fj_*\EE j^{-1} \DRE(\cM)\big)&\iso \RR{\varpi_{D,f}}_*\shbuD\big(\EE \Dj_*\EE j^{-1} \DRE(\cM)\big)\\
		&\iso \RR{\varpi_{D,f}}_*\DRmodD(\cM)\iso \DRmodf(\cM).
	\end{align*}
Here, the first isomorphism follows from \cite[Lemma 3.9]{DK21sh} since $\varpi_{D,f}$ is proper. The last isomorphism is proved in \cite[Proposition 4.7.4]{Moc14}.
\end{proof}

\subsection{Moderate growth de Rham complexes along a function}
In the previous section, we have studied the moderate growth de Rham complex on real blow-ups $\XD$ and $\Xf$ in the case of a normal crossing divisor. To do this, we could directly apply the constructions performed in \cite{DK16}. 

If $X$ is a complex manifold and $f\colon X\to \C$ is a holomorphic function, the divisor $f^{-1}(0)$ does not need to have normal crossings. The blow-up space $\Xf$ can still be defined, but Corollary~\ref{cor:differentblowups} does not apply to this case. In this subsection, we define a version of the enhanced de Rham functor on $\Xf$ in order to prove an analogous statement without the normal crossing assumption. This works along the same lines as the version for $\XD$ in \cite{DK16} (simply denoted by $\wt{X}$ in loc.~cit.), but uses some interesting facts about resolutions of singularities.

Throughout this subsection, $X$ will be a complex manifold and $f\colon X\to \C$ a holomorphic function. We write $X^*\vcentcolon= X\setminus f^{-1}(0)$. We denote by $\Xf$ the real blow-up of $X$ along $f$. Recall the notation for morphisms from \eqref{eq:morphXf}.
We set $\cD_{\Xf}^\cA\vcentcolon= \Amodf\otimes_{\varpi_f^{-1}\cO_X} \varpi_f^{-1}\cD_X$. Let $\overline{X}$ denote the complex manifold conjugate to $X.$

The first step is to define an ind-sheaf of \emph{tempered holomorphic functions} on $\Xf$, by analogy with \cite[§7.2]{DK16}.

\begin{lemmadefi}\label{lemmadefi:OtXf}
We define the objects
$$\Dbt{\Xf}\vcentcolon=\varpi_f^! \RIhom(\C_{X^*}, \Dbt{X})$$
and
$$\cO^\mathrm{t}_{\Xf}\vcentcolon= \RR \cH om_{\varpi_f^{-1}\cD_{\overline{X}}}\big(\varpi_f^{-1}\cO_{\overline{X}},\Dbt{\Xf}\big),$$
and the latter is an object in $\mathrm{D}^\mathrm{b}(\beta\cD_{\Xf}^\cA)$.
\end{lemmadefi}
\begin{proof}

It is clear a priori that $\Dbt{\Xf}$ thus defined is an object over $\mathrm{I}(\varpi_f^{-1}\cD_X\otimes\varpi_f^{-1}\cD_{\overline{X}})$, and hence that $\cO^\mathrm{t}_{\Xf}$ is a module over $\varpi_f^{-1}\cD_X$. It thus remains to show that it is an $\Amodf$-module.

\begin{itemize} 
\item[(a)] Set $D\defeq f^{-1}(0)$. Assume that we can find a modification $\tau\colon Y\to X$ (i.e., a proper morphism such that $E\defeq \tau^{-1}(D)$ has simple normal crossings and induces an isomorphism $Y\setminus E\iso X\setminus D$) and set $g\defeq f\circ \tau$. Let us write $X^*\defeq X\setminus D$ and $Y^*\vcentcolon=Y\setminus E$. Let $\varpi_{E,g} \colon \wt Y_E \to \wt Y_g$ denote the natural proper morphism between the two real blow-ups on $Y$ associated to $E=g^{-1}(0)$, denote by $\wt{\tau}\colon\wt{Y}_g\to \Xf$ the map induced by $\tau$, and let $\wt \tau_g = \wt \tau \circ \varpi_{E,g}\colon \wt Y_E \to \Xf$ denote the composition. We have the following commutative diagram, where the square is Cartesian:
\begin{equation}\label{eq:diagramResSing}
    \begin{tikzcd}
\wt{Y}_E\arrow{rr}{\varpi_{E,g}}\arrow[bend left]{rrrr}{\wt{\tau}_g}\arrow{rrdd}[swap]{\varpi_E} && \wt{Y}_g\arrow{rr}{\wt{\tau}}\arrow{dd}{\varpi_g}\arrow[phantom, ddrr, "\square"] && \Xf\arrow{dd}{\varpi_f} \\ \\
&& Y\arrow{rr}{\tau} && X
\end{tikzcd}
\end{equation}
Since $E$ is a normal crossing divisor, we know from \cite[Notation 7.2.4 and Theorem 7.2.7]{DK16} that
$$\Ot{\wt{Y}_E}\iso \varpi_E^!\RIhom(\C_{Y^*},\Ot{Y})$$
is an $\cA_{\wt{Y}_E}^\mathrm{mod}$-module. 

Moreover, we have isomorphisms
\begin{align*}
    \RR  \ttaug_* \Ot{\wt{Y}_E} &\iso \RR \ttau_* \RR{\varpi_{E,g}}_* \RIhom(\C_{Y^*},\varpi_E^!\Ot{Y})\\
    &\iso \RR \ttau_* \RR{\varpi_{E,g}}_* \RIhom(\C_{Y^*},\varpi_{E,g}^!\varpi_g^!\Ot{Y})\\
    &\iso \RR \ttau_* \RIhom(\RR{\varpi_{E,g}}_*\C_{Y^*},\varpi_g^!\Ot{Y})\\
    &\iso \RR \ttau_* \RIhom(\C_{Y^*},\varpi_g^!\Ot{Y})\\
    &\iso \RIhom(\C_{X^*},\RR \ttau_* \varpi_g^!\Ot{Y})\\
    &\iso \RIhom(\C_{X^*},\RR \varpi_f^!\tau_*\Ot{Y})\\
    &\iso \varpi_f^!\RIhom(\C_{X^*},\RR \tau_*\Ot{Y})\\
    &\iso \varpi_f^!\RIhom(\C_{X^*},\Ot{X})\iso \Ot{\Xf},
\end{align*}
Here, we repeatedly use adjunction isomorphisms like \cite[Proposition 5.3.8, Corollary 5.3.5]{KS01} (together with the fact that all the morphisms in \eqref{eq:diagramResSing} are isomorphisms outside the given divisors). The sixth isomorphism uses the \cite[Theorem 5.3.10]{KS01}, relying on the fact that the square in \eqref{eq:diagramResSing} is Cartesian. The last line follows from the tempered Grauert theorem \cite[Theorem 3.1.5]{KS16}.

Consequently, $\Ot{\Xf}$ is a module (or, in general, a complex of modules) over $\ttaug_* \cA_{\wt{Y}_E}^\mathrm{mod}\iso \RR \ttaug_* \cA_{\wt{Y}_E}^\mathrm{mod}\iso \Amodf$. This isomorphism follows from \cite[Theorem 4.1.5]{Moc14}. (Let us be very precise here, emphasizing the functor $\beta$ that is usually suppressed in the notation: Taking the direct image, it becomes a module over $\ttaug_* \beta \cA_{\wt{Y}_E}^\mathrm{mod}$, and this induces a $\beta \ttaug_*\cA_{\wt{Y}_E}^\mathrm{mod}$-module structure by the natural morphism $\beta \ttaug_*\to \ttaug_* \beta$, cf.\ \cite[Proposition 4.3.17]{KS01}, noting that $\ttaug$ is proper.)

\item[(b)]
Let us show that the action of $\Amodf$ on $\Ot{\Xf}$ constructed in (a) is canonical and does not depend on the choice of the modification $\tau\colon Y\to X$.

First, we note that, similarly to the computation in (a), one obtains
\begin{align*}
\wt \tau_g^! \Dbt{\Xf} &= \wt \tau_g^! \varpi_f^!\RIhom(\C_{X^*},\Dbt{X}) \\
&\iso \varpi_E^! \tau^! \RIhom(\C_{X^*},\Dbt{X}) \\
&\iso \varpi_E^! \RIhom(\C_{Y^*},\Dbt{Y}) \\
&\iso \Dbt{\wt Y_E}.
\end{align*}
The second-to-last isomorphism is due to \cite[Lemma 2.5.7]{KS16}, and the last isomorphism follows from \cite[Theorem 7.2.7]{DK16} (since $E$ is a normal crossing divisor).
Consequently, one has an isomorphism $\Dbt{\wt{Y}_E}\iso \wt{\tau}_g^! \RR\ttaug_* \Dbt{\wt{Y}_E}$.

Now, observe that
\begin{align*}
    \Ot{\Xf}&\iso \RR\ttaug_*\Ot{\wt{Y}_E}\iso \RR\ttaug_*\RR\cH om_{\varpi_E^{-1}\cD_{\overline{Y}}}(\varpi_E^{-1}\cO_{\overline{Y}},\Dbt{\wt{Y}_E})\\
    &\iso \RR\ttaug_*\RR\cH om_{\varpi_E^{-1}\cD_{\overline{Y}}}(\varpi_E^{-1}\cO_{\overline{Y}},\wt{\tau}_g^!\RR\ttaug_*\Dbt{\wt{Y}_E})\\
    &\iso \RR\cH om_{\ttaug_*\varpi_E^{-1}\cD_{\overline{Y}}}(\RR\ttaug_*\varpi_E^{-1}\cO_{\overline{Y}},\RR\ttaug_*\Dbt{\wt{Y}_E}),
\end{align*}
showing that the $\Amodf$-action on $\Ot{\Xf}$ is induced by the $\Amodf$-action on $\RR\ttaug_*\Dbt{\wt{Y}_E}$, which in turn is induced by the $\cA^{\mathrm{mod}}_{\wt{Y}_E}$-action on $\Dbt{\wt{Y}_E}$. We remark that $\RR\ttaug_*\Dbt{\wt{Y}_E}=\ttaug_*\Dbt{\wt{Y}_E}$ since tempered distributions form a quasi-injective object (cf.\ \cite[Theorem 3.18]{Kas84}, \cite[§7.2]{KS01}).

In fact, the action of $\Amodf\iso \ttaug_*\cA^{\mathrm{mod}}_{\wt{Y}_E}$ on $\ttaug_*\Dbt{\wt{Y}_E}$ does not depend on the choice of the projective morphism $\tau$ above: Let $U\subset \Xf$ be a subanalytic subset. A local section of $\Amodf$ on $U$ is a holomorphic function, say $\varphi$, on $U\cap X^*=\tau^{-1}(U)\cap Y^*$ that has tempered growth along $U\cap \partial \Xf$ (or along $\tau^{-1}(U)\cap \partial \widetilde{Y}_E$, which is equivalent, since $\ttaug_*\cA^\mathrm{mod}_{\widetilde{Y}_E}\iso \Amodf$). On the other hand, we have $(\ttaug_*\Dbt{\wt{Y}_E})(U)=\Dbt{\widetilde{Y}_E}(\widetilde{\tau}_g^{-1}(U))$, so an element $\delta$ thereof is by definition a (tempered) distribution on $\widetilde{\tau}_g^{-1}(U)\cap Y^*=U\cap X^*$. The action of $\varphi$ on $\delta$ is given by the natural multiplication of a distribution by a function, all defined outside the divisors. Hence, the induced action is independent of the choice of $\tau$.
\item[(c)] In general, a modification similar to that in (a) exists by now well-known results on desingularization of analytic spaces (see, e.g., \cite{AHV18}). However, $E$ might not have smooth components, but we know that, locally around each point $x\in X$, we can find such a modification. Therefore, (a) shows that $\Ot{\Xf}$ has locally an action of $\Amodf$. Since the local actions are canonical (and hence compatible) by (b), they give a global action of $\Amodf$ on $\Ot{\Xf}$.
\end{itemize}
\end{proof}

We can now define (using notation similar to the one in \cite[§9.2]{DK16})
\begin{align*}
    \cO_{\Xf}^\EE&\vcentcolon= \tilde{i}^!\RR\cH om_{\cD_\PP}(\cE^\tau_{\C|\PP}, \cO_{\Xf\times \PP}^\mathrm{t})[2],\\
   \Omega_{\Xf}^\EE &\vcentcolon= \pi_{\Xf}^{-1}\varpi_f^{-1}\Omega_X\otimes^\LL_{\pi_{\Xf}^{-1}\varpi_f^{-1}\cO_X} \cO_{\Xf}^\EE
\end{align*}

and one sets
$$\DREbuf(\cL)\vcentcolon=\Omega_{\Xf}^\EE\otimes^\LL_{\cD_{\Xf}^\cA}\cL \qquad \textnormal{for $\cL\in\mathrm{D}^\mathrm{b}(\cD_{\Xf}^\cA)$}.$$
As in loc.~cit.\ we then get
$$\DREbuf(\cM^{\cA_f})\iso \EE\varpi_f^!\DRE(\cM(*D))$$
for any $\cM \in \DbholD{X}$ (setting $\cM^{\cA_f}\vcentcolon=\Amodf\otimes_{\varpi_f^{-1}\cO_X}\varpi_f^{-1}\cM$).

Then we can reproduce the proof of Lemma~\ref{lemmaShDR} and Proposition~\ref{prop:DRmod} along the exact same lines and obtain the following.
\begin{prop}\label{prop:DRmodf}
	Let $X$ be a complex manifold, $f\colon X\to \C$ a holomorphic function and $\cM\in\DbholD{X}$. Then there is an isomorphism in $\DbC{\Xf}$
	$$\shbuf\big(\EE \fj_*\EE j^{-1} \DRE(\cM)\big)\iso \DRmodf(\cM).$$
\end{prop}
\begin{proof}
The main steps of this proof are the isomorphisms 
$$\cO^\EE_{\Xf}\iso\RIhom(\wt{\varpi}_f^{-1}\pi^{-1}\C_{X^*},\EE\varpi_f^!\cO^\EE_X)$$
and
$$\EE\varpi_f^{-1}\SolE\big(\cO_X(*D)\big)\iso\mathcal{S} ol^\EE_{\Xf}\big(\cO_X(*D)^{\cA_f}\big),$$
which we both derive from Lemma-Definition~\ref{lemmadefi:OtXf} as in the case of $\XD$ in \cite{DK16}, and the isomorphism
$$\alpha_{\Xf}\Ot{\Xf}\iso\Amodf,$$
and this follows again from a local consideration as in part (a) of the above proof since $\alpha_{\Xf}$ is compatible with direct images (see \cite[Proposition 4.3.6]{KS01}) and we know such a statement for the blow-up $\wt{Y}_E$ along a normal crossing divisor (see \cite[Proposition 7.2.10]{DK16}).
\end{proof}

\subsection{Rapid decay de Rham complexes: a motivation in dimension one}\label{subsec:rapiddecay}
So far we have studied moderate growth de Rham complexes. The aim is now to find a functorial way to extract the rapid decay de Rham complex from the enhanced de Rham complex. Rapid decay functions do not seem to have been studied in the context of enhances ind-sheaves.

Studying the one-dimensional case again, it is quite straightforward to find an expression similar to the one for the moderate growth de Rham complex. Let $X$ be a complex manifold. Recall the notation from Section~\ref{sec:motivation}.

\begin{lemma}\label{lemmaTriangle}For an object $K\in\EbIC{X}$, one has a distinguished triangle in $\EbIC{\wt{X}}$
	$$\EE \tj_{!!}\EE j^{-1} K \To \EE \tj_*\EE j^{-1} K\To \EE \tilde i_{!!}\EE \tilde i^{-1}\EE \tj_*\EE j^{-1} K\ToPO$$
\end{lemma}
\begin{proof}
	By \cite[Lemma 2.7.6]{DK19}, we know that there are isomorphisms
	\begin{align*}
		\EE \tj_{!!}\EE \tj^{-1} H &\iso \pi^{-1}\C_{X\setminus D}\otimes H\\
		\EE \tilde i_{!!}\EE \tilde i^{-1} H &\iso \pi^{-1}\C_{\partial\wt{X}}\otimes H
	\end{align*}
	for any $H\in\EbIC{\wt{X}}$.
	
	From the natural short exact sequence (cf.\ \cite[Proposition 2.3.6(v)]{KS90})
	$$0\To \C_{X\setminus D}\To \C_{\wt{X}}\To \C_{\partial\wt{X}}\To 0$$
	in $\DbC{\wt{X}}$, we therefore obtain (applying the functor $\pi^{-1}(\bullet)\otimes \EE \tj_*\EE j^{-1} K$) a distinguished triangle
	$$\EE \tj_{!!}\EE \tj^{-1} \EE \tj_* \EE j^{-1} K\To \EE \tj_*\EE j^{-1}K\To \EE \tilde i_{!!}\EE \tilde i^{-1} \EE \tj_* \EE j^{-1} K\ToPO$$
	and this is the desired triangle, noting that $\EE j^{-1} \EE j_*\iso \EE j^! \EE j_* \iso \id$.
\end{proof}

When $K$ additionally is a \emph{perverse} enhanced ind-sheaf (that is, an element of the essential image of $\mathrm{Mod}_{hol}(\cD_X)$ under the enhanced de Rham functor; see \cite[§4]{DK20cycles} for more detail on this notion---we will not need it elsewhere in this paper), we can make the following, more specific, claim.

\begin{prop}\label{prop:ExactSeq}
Let $K\in\EbIC{X}$ be a perverse enhanced ind-sheaf, then there is a short exact sequence
$$0\To \tilde i^{-1}\shbu(\EE\tj_{!!}\EE j^{-1}K)[-1] \To \Psi^{\leq 0}_a(K) \To \Psi^0_a(K)\To 0.$$
\end{prop}
\begin{proof}
We apply the functor $\tilde i^{-1}\shbu(\bullet)$ to the distinguished triangle from Lemma \ref{lemmaTriangle}. Then the middle object becomes $\Psi^{\leq 0}_a(K)$. Moreover, the third object becomes $\Psi^0_a(K)$, noting that $\shbu\circ \EE \tilde i_{!!}\iso \tilde i_! \circ \sh{\partial \wt{X}}$ since $\tilde i$ is proper (see \cite[Lemma 3.9(ii)]{DK21sh}) and $\tilde i^{-1}\tilde i_!\iso \id$.

The distinguished triangle thus obtained is indeed a short exact sequence since we already know that the second and third objects are concentrated in degree zero and the morphism between them is an epimorphism (see \cite[Lemma 4.3]{DK20cycles}), which implies that also the first object must be concentrated in degree $0$.
\end{proof}

Now let $\cM$ a meromorphic connection with a pole at $0$, and let $K=\DRE(\cM)$ and $\cL_\cM = \tilde i^{-1}\tj_*j^{-1}\DR(\cM)[-1]$. Then the objects in the short exact sequence of Proposition~\ref{prop:ExactSeq} are related to the Stokes filtration by
\begin{align*}
\Psi^{\leq 0}_a(K)&\iso \mathrm{F}_{\leq 0}\cL_{\cM},\\
\Psi^0_a(K) &\iso \mathrm{Gr}^\mathrm{F}_0 \cL_{\cM}=\mathrm{F}_{\leq 0}\cL_\cM / \mathrm{F}_{<0}\cL_\cM.
\end{align*}
Therefore, this implies that there is an isomorphism
$$\mathrm{F}_{<0}\cL_{\cM} \iso \tilde i^{-1}\shbu(\EE\tj_{!!}\EE j^{-1}K)[-1].$$
Since the ``$< 0$'' part of the Stokes filtration is nothing but the rapid decay de Rham complex, restricted to $\partial \wt{X}$, this implies that it is reasonable to expect an isomorphism
\begin{equation}\label{eq:ideaDRrd}
\DRrd(\cM)\iso \shbu\big(\EE\tj_{!!}\EE j^{-1}\DRE(\cM)\big),
\end{equation}
a formula similar to the ones proved for moderate growth de Rham complexes above.

\begin{rem}
Note that the functor $\shbu \circ \EE \tj_{!!}\not\simeq \tj_!\circ \sh{X^*_\infty}$, in particular this functor is not an ``extension by zero'' as known from the theory of classical sheaves.
\end{rem}
 
A proof of this formula \eqref{eq:ideaDRrd} in the case that $f$ defines a normal crossing divisor (in any dimension) will be given in Section~\ref{sec:TWduality} (see Corollary~\ref{cor:DRrdXD}), using deep results on duality of certain sheaves of functions.

\section{Moderate growth and rapid decay objects associated to enhanced ind-sheaves}\label{sec:enhmodrd}

In Section~\ref{sec:DRmodandrd}, we have studied functorial ways to extract the moderate growth and rapid decay de Rham complexes from the enhanced de Rham complex. Inspired by these considerations, we define here moderate growth and rapid decay objects (more precisely, sheaves on the real blow-up space) to any enhanced ind-sheaf on $X$ and establish some of their basic properties.

Throughout this section, let $X$ be a complex manifold and let $f\colon X\to \C$ be  holomorphic function. Since we will only work on the blow-up $\Xf$ from here on, we will simplify the notation for the morphisms from diagram \eqref{eq:morphXf} by suppressing the index $f$ as follows:
\begin{equation*}
\begin{tikzcd}\partial\Xf\arrow[hook, r, "\tilde i"] & \Xf\arrow[bend left=40,rr,"\varpi"]&X^*_\infty\arrow[l,"\tj",swap,hook']\arrow[r,hook,"j"] & X\end{tikzcd}
\end{equation*}
We will also work over an arbitrary field $k$ here.

\begin{defi}\label{def:enhmodrd}
For an enhanced ind-sheaf $K\in\EbIk{X}$, we define the objects of $\Dbk{\Xf}$
\begin{align}
    K^{\mathrm{mod}\, f}&\defeq \shbuf\big(\EE \tj_*\EE j^{-1} K\big),\\
    K^{\mathrm{rd}\, f}&\defeq \shbuf\big(\EE \tj_{!!}\EE j^{-1} K\big).
\end{align}
\end{defi}

\begin{rem}
Let us note here that we could also perform the same construction on the real blow-up $\XD$ along a normal crossing divisor and in this way define objects $K^{\mathrm{mod}\, D}$ and $K^{\mathrm{rd}\, D}$. We will, however, focus on the blow-up along a function from here on, although the constructions and results we prove in the present and the following section work analogously on $\XD$. We will use this fact in the proofs of Proposition~\ref{prop:conjSabNCD} and Proposition~\ref{cor:ModRdpairing1} below.
\end{rem}

The first easy observation is the following.

\begin{lemma}
   There are isomorphisms for any $K\in\EbIk{X}$
    \begin{align*}
        K^{\mathrm{mod}\, f} &\iso (\pi^{-1}k_{X^*}\otimes K)^{\mathrm{mod}\, f} \iso \RIhom(\pi^{-1}k_{X^*},K)^{\mathrm{mod}\, f},\\
        K^{\mathrm{rd}\, f} &\iso (\pi^{-1}k_{X^*}\otimes K)^{\mathrm{rd}\, f} \iso \RIhom(\pi^{-1}k_{X^*},K)^{\mathrm{rd}\, f}.
    \end{align*}
\end{lemma}
\begin{proof}
It follows from \cite[Lemma 2.7.6]{DK19} that we have
\begin{align*}
    \pi^{-1}k_{X^*}\otimes K \iso \EE j_{!!} \EE j^{-1} K,\\
    \RIhom(\pi^{-1}k_{X^*},K) \iso \EE j_* \EE j^{-1} K.
\end{align*}
(For the second isomorphism, note that $j$ is an open embedding and hence $j^!=j^{-1}$.) Then, the statement of the lemma is clear from the definition because $\EE j^{-1} \EE j_*=\id=\EE j^{-1} \EE j_{!!}$.
\end{proof}

\begin{rem}
In the context of enhanced de Rham functors (in particular, $k=\C$ here), the previous lemma tells us that the functors from Definition~\ref{def:enhmodrd} are not sensitive to localization of the $\cD_X$-module: If $D=f^{-1}(0)$ and $\cM\in\DbholD{X}$, then by \cite[Theorem 9.1.2]{DK16}, we know that
\begin{align*}
    \DRE\big(\cM(*D)\big) &\iso \RIhom\big(\pi^{-1}\C_{X^*},\DRE(\cM)\big),\\
    \DRE\big(\cM(!D)\big) &\iso \pi^{-1}\C_{X^*}\otimes\DRE(\cM)
\end{align*}
and hence the lemma means nothing but
\begin{align*}
        \big(\DRE(\cM)\big)^{\mathrm{mod}\,f}&\iso \big(\DRE(\cM(*D))\big)^{\mathrm{mod}\,f}\iso \big(\DRE(\cM(!D))\big)^{\mathrm{mod}\,f},\\
        \big(\DRE(\cM)\big)^{\mathrm{rd}\,f}&\iso \big(\DRE(\cM(*D))\big)^{\mathrm{rd}\,f}\iso \big(\DRE(\cM(!D))\big)^{\mathrm{rd}\,f}.
    \end{align*}
\end{rem}

By restricting the objects from Definition~\ref{def:enhmodrd} to the boundary of the real blow-up, we can define objects that will be called \emph{enhanced nearby cycles}; they are higher-dimensional analogues of the constructions in \cite{DK20cycles} as well as analogues in the world of enhanced ind-sheaves of the constructions performed for $\cD_X$-modules in \cite{Sab21}.

\begin{defi}\label{def:Cycles}
    Let $K\in \EbIk{X}$, then we set
    \begin{align*}
        \psi_f^{\leq 0} K &\vcentcolon= \tilde i^{-1}K^{\mathrm{mod}\, f}= \tilde i^{-1} \shbuf\big(\EE \tj_*\EE j^{-1} K\big)[-1],\\
        \psi_f^{< 0} K &\vcentcolon= \tilde i^{-1}K^{\mathrm{rd}\, f} = \tilde i^{-1} \shbuf\big(\EE \tj_{!!}\EE j^{-1} K\big)[-1],\\
        \psi_f^{> 0} K &\vcentcolon= \tilde i^{!}K^{\mathrm{mod}\, f} =\tilde i^! \shbuf\big(\EE \tj_*\EE j^{-1} K\big),\\
        \psi_f^{\geq 0} K &\vcentcolon= \tilde i^{!}K^{\mathrm{rd}\, f} = \tilde i^! \shbuf\big(\EE \tj_{!!}\EE j^{-1} K\big).
    \end{align*}
    Moreover, set
    $$ \psi_f^* K \vcentcolon = \tilde i^{-1} \RR \tj_* j^{-1} \sh{X} K[-1].$$
    All of them are objects in $\Dbk{\partial\Xf}$.
\end{defi}

\begin{rem}
One can express these objects differently in the case where $K$ is the enhanced de Rham or solution object of a meromorphic connection (and in particular $k=\C$).

For example, let $\cM$ be a meromorphic connection with poles along the normal crossing divisor $D=f^{-1}(0)$. Then
\begin{align*}
    \psi_f^{\leq 0}\DRE(\cM) &\cong \tilde i^{-1} \shbuf\big(\EE \tj_*\EE j^{-1} \DRE(\cM)\big)[-1]\\
    &\cong \tilde i^{-1} \shbuf\big(\EE \varpi^! \EE j_*\EE j^{-1} \DRE(\cM)\big)[-1]\\
    &\cong \tilde i^{-1} \shbuf\big(\EE \varpi^! \DRE(\cM)\big)[-1].
\end{align*}

Similarly, one can show $\psi_f^{>0}\DRE(\cM)\iso \tilde i^! \shbuf\big(\EE \varpi^! \DRE(\cM)\big)$, which is the analogue of the complex ${}^p\psi_f^{>\mathrm{mod}} \cM$ in \cite{Sab21}.

Along the same lines, we get $\psi_f^{<0}\SolE(\cM)\iso \tilde i^{-1} \shbuf\big(\EE \varpi^{-1} \SolE(\cM)\big)[-1]$ and $\psi_f^{\geq 0}\SolE(\cM)\iso \tilde i^! \shbuf\big(\EE \varpi^{-1} \SolE(\cM)\big)$.

\end{rem}

The following lemma is clear by construction.
\begin{lemma}\label{lemma:constructible}
    For $K\in\EbRcIk{X}$ (in particular for $K=\DRE(\cM)$ for a holonomic $\cD_X$-module in the case $k=\C$), the objects defined in Definition~\ref{def:Cycles} are $\R$-constructible.
\end{lemma}
\begin{proof}
It suffices to note that the functors $\EE\fj_*$, $\EE \fj_{!!}$, $\EE j^{-1}$ and $\shbuf$ involved in the construction of $K^{\mathrm{mod}\, f}$ and $K^{\mathrm{rd}\, f}$ all preserve $\R$-constructibility (see \cite[Proposition 3.3.3]{DK19} and \cite[Theorem 6.6.4]{KS16}).
\end{proof}

\begin{prop}\label{prop:fundamentaltriangles}
We have the following natural distinguished triangles in the category $\Dbk{\partial \Xf}$:
$$\begin{tikzcd}
    \psi_f^{\leq 0}K\arrow[r,"v"]&\psi_f^* K\arrow{r} &\psi_f^{>0}K \arrow{r}{+1} & {}\\[-10pt]
    \psi_f^{< 0}K\arrow{r}& \psi_f^* K\arrow[r,"c"] & \psi_f^{\geq 0}K \arrow{r}{+1} & {}
\end{tikzcd}$$
\end{prop}

\begin{proof}
We start with the natural distinguished triangle
\begin{equation}\label{eqn:ModTriangle}
\shbuf(\EE \tj_*\EE j^{-1} K) \to R\tj_*\tj^{-1}\shbuf(\EE \tj_*\EE j^{-1} K) \to \tilde i_! \tilde i^!\shbuf(\EE \tj_*\EE j^{-1} K)[1] \xrightarrow{+1}
\end{equation}
(cf., e.g., \cite[Proposition 2.4.6]{KS90}).
Applying $\tilde i^{-1}[-1]$ and noting that we have natural isomorphisms (recall the compatibility of the sheafification functor with pullbacks along open embeddings from \cite[Lemma 3.9]{DK21sh})
\begin{align*}
\tilde i^{-1}R\tj_*\tj^{-1}\shbuf(\EE \tj_*\EE j^{-1} K) &\iso \tilde i^{-1}R\tj_*\sh{X^*_\infty}(\EE \tj^{-1}\EE \tj_*\EE j^{-1} K) \\
&\iso \tilde i^{-1}R\tj_* \sh{X^*_\infty}(\EE j^{-1} K) \\
&\iso \tilde i^{-1}R\tj_*j^{-1}\sh{X}(K),
\end{align*}
we obtain the first distinguished triangle.

The second triangle is found analogously, replacing $\EE \tilde j_*$ by $\EE \tilde j_{!!}$ in \eqref{eqn:ModTriangle}.
\end{proof}

\begin{prop}
There is a natural distinguished triangle in $\Dbk{\partial\Xf}$
$$\psi_f^{< 0}K \To \psi_f^{\leq 0} K \To \psi_f^{0} K \ToPO$$
where $\psi_f^0K\vcentcolon= \mathsf{sh}_{\partial \Xf}(\EE \tilde i^{-1}\EE \tj_* \EE j^{-1} K)$.
\end{prop}
\begin{proof}
This follows directly from Lemma~\ref{lemmaTriangle}.
\end{proof}

Let $p\colon X \to Y$ be a proper holomorphic map, $g\colon Y \to \C$ any holomorphic map, and set $f \defeq g \circ p$. Then, $p$ lifts to a morphism between the real blow-up spaces $\wt{p}\colon \Xf = X \times_Y \wt{Y}_g \rightarrow \wt{Y}_g$, and we obtain a natural commutative diagram where all the squares are Cartesian:
\begin{equation}\label{eq:basechangediagram}
\begin{tikzcd}
\partial \wt X_f \ar[rr,"i_f"] \ar[dd,"\wt{p}_0"] \ar[phantom, ddrr, "\square"] && \wt X_f \ar[rr,"\varpi_f"] \ar[dd,"\wt{p}"] \ar[phantom, ddrr, "\square"] && X \ar[dd,"p"] \ar[phantom, ddrr, "\square"] && X^* \ar[ll,swap,"j_f"] \ar[bend right, llll, swap, "\fj"] \ar[dd,swap,"p|_{X^*}"]\\ \\
\partial \wt{Y}_g \ar[rr,"i_g"] && \wt{Y}_g \ar[rr,"\varpi_g"] && Y && Y^* \ar[ll,swap,"j_g"] \ar[bend left, llll, "\tilde{j}_g"]
\end{tikzcd}
\end{equation}

We have the following statement, which is the analogue of \cite[Proposition 6.6]{Sab21}.

\begin{prop} Let $X,Y,f,g$, and $p$ be as above, and let $K \in \EbIk{X}$. Then there is an isomorphism of distinguished triangles in $\mathrm{D}_{\R\text{-}\mathrm{c}}^{\mathrm{b}}(k_{\partial \wt Y_g})$:
\begin{equation*}
\begin{tikzcd}
\RR \tpzero_* \psi_f^{\leq 0} K \arrow{r} \arrow{d}{\simeq} & \RR \tpzero_*  \psi_f^* K \arrow{r} \arrow{d}{\simeq} & \RR \tpzero_* \psi_f^{> 0} K \arrow{r}{+1} \arrow{d}{\simeq} & {}\\
\psi_g^{\leq 0} \EE p_* K \arrow{r} &  \psi_g^* \EE p_* K \arrow{r} & \psi_g^{> 0} \EE p_* K \arrow{r}{+1} & {}
\end{tikzcd}
\end{equation*}
Similarly, there is an isomorphism
\begin{equation*}
\begin{tikzcd}
\RR \tpzero_* \psi_f^{< 0} K \arrow{r} \arrow{d}{\simeq} & \RR \tpzero_* \psi_f^* K \arrow{r} \arrow{d}{\simeq} & \RR \tpzero_* \psi_f^{\geq 0} K \arrow{r}{+1} \arrow{d}{\simeq} & {}\\
\psi_g^{< 0} \EE p_* K \arrow{r} &  \psi_g^*  \EE p_* K \arrow{r} & \psi_g^{\geq 0} \EE p_* K \arrow{r}{+1} & {}
\end{tikzcd}
\end{equation*}
\end{prop}
\begin{proof}
Since all the isomorphisms work along the same lines, let us only prove the first one (on the left of the first isomorphism of triangles).
\begin{align*}
    \RR\tpzero_* \psi_f^{\leq 0} K &\iso \RR\tpzero_! \tilde i_f^{-1} \sh{\Xf}\big( \EE \fj_*\EE j_f^{-1} K)[-1]\\
    &\iso \tilde i_g^{-1} \RR\wt{p}_! \sh{\Xf}\big( \EE \fj_*\EE j_f^{-1} K)[-1]\\
    &\iso \tilde i_g^{-1} \sh{\Xf}\big( \EE\wt{p}_{!!} \EE \fj_*\EE j_f^{-1} K)[-1]\\
    &\iso \tilde i_g^{-1} \sh{\Xf}\big( \EE \gj_* \EE{(p|_{X^*})}_{!!} \EE j_f^{-1} K)[-1]\\
    &\iso \tilde i_g^{-1} \sh{\Xf}\big( \EE \gj_* \EE j_g^{-1} \EE p_{!!} K)[-1] \iso \psi_g^{\leq 0} \EE p_* K.
\end{align*}
Here, we repeatedly used the facts that the squares in \eqref{eq:basechangediagram} are Cartesian (see \cite[Proposition 4.5.11]{DK16} for a base change formula for operations on enhanced ind-sheaves) and that $p$ is proper and hence $\EE p_*=\EE p_{!!}$ (and similarly for $\wt{p}$).

The isomorphisms of objects thus obtained induce isomorphisms of distinguished triangles since the morphisms in the triangles of Proposition~\ref{prop:fundamentaltriangles} are all canonical and the isomorphism constructed above is natural.
\end{proof}

\section{Local and global duality on the real blow-up}\label{sec:duality}

\subsection{A short review on duality and pairings in derived categories}

Let $k$ be a field and let $M$ be a good topological space. We recall the following notion of a perfect pairing in the derived category $\Dbk{M}$. A good reference is \cite{KS90}, as well as \cite[Appendix C]{FSY21}.

Let $F,G \in \Dbk{M}$, and let $\omega_M$ denote the \emph{Verdier dualizing complex}. If $a_M\colon M \to \{\mathrm{pt}\}$ is the canonical map to the one-point space, then we have $\omega_M = a_M^!k$. The \emph{Verdier dual} of $F \in \Dbk{M}$ is the object $\mathrm{D}_M F \vcentcolon= \sHom(F,\omega_M)$. 

\begin{defi}\label{def:pairing}
Recall that, for $F,G\in\DbRck{X}$, a \emph{pairing} $\eta\colon F \otimes_{k_M} G \to \omega_M$ is equivalent to the datum of a morphism $F \to \mathrm{D}_M G$ in $\Dbk{M}$ (or, equivalently, a morphism $G \to \mathrm{D}_M F$).

We say that a pairing $\eta\colon F \otimes_{k_M} G \to \omega_M$ is \emph{perfect} if the associated morphism $F \to \mathrm{D}_M G$ (or $G \to \mathrm{D}_M F$) is an isomorphism. 
\end{defi}

Now, suppose $\eta\colon F \otimes_{k_M} G \to \omega_M$ is such a perfect pairing with an isomorphism $G \iso D_M F$. Taking derived global sections, we find 
\begin{align*}
\RGamma(M;\mathrm{D}_M F) &\iso \Hom(F,\omega_M) \\
&\iso \Hom(\RGamma_c(M;F),k) \\
&=\vcentcolon \RGamma_c(M;F)^\vee.
\end{align*}
In particular, the hypercohomology of this object satisfies, for all $\ell \in \Z$,
\begin{align*}
\hyp^\ell(M;\mathrm{D}_M F) \iso \mathrm{H}^0 \RGamma(M;(\mathrm{D}_M F) [\ell]) &\iso \mathrm{H}^0 \Hom(\RGamma_c(M;F),k)[\ell] \\
&\iso \mathrm{H}^0 \Hom(\RGamma_c(X;F)[-\ell],k) \\  
&\iso \hyp_c^{-\ell}(M;F)^\vee.
\end{align*}

By \cite[(2.6.23)]{KS90}, there is a natural morphism 
$$
\RGamma(M;\mathrm{D}_M F)\otimes_k \RGamma_c(M;F) \to \RGamma_c(M;F \otimes_{k_M} \mathrm{D}_M F).
$$
Taking hypercohomology and composing with the morphism induced by $\eta$ on global sections, we obtain pairings for all $\ell \in \Z$:
\begin{equation}\label{eqn:globalpairing}
\hyp^\ell(M;\mathrm{D}_M F) \otimes_k \hyp_c^{-\ell}(M;F) \longrightarrow \hyp_c^0(M;\omega_M) \iso k . 
\end{equation}

\begin{prop}[{\cite[Corollary C.6]{FSY21}}]\label{prop:pairingref}
Let $F,G \in \Dbk{M}.$ If the pairing 
\begin{align*}
\eta: F \otimes_{k_M} G \to \omega_M
\end{align*}
is perfect, then so are the pairings
\begin{align*}
\eta^\ell: \hyp^\ell(M;F) \otimes_k \hyp_c^{-\ell}(M;G) \to k
\end{align*}
for all $\ell \in \Z.$
\end{prop}

We will refer to the pairings (\ref{eqn:globalpairing}) as the global duality pairings of $F$ (on hypercohomology). Similarly, we will refer to $\eta\colon F \otimes_{k_M} \mathrm{D}_M F \to \omega_M$ as the local duality pairing of $F$. 

\subsection{Local duality statements for enhanced nearby cycles}
With our definition of the moderate growth and rapid decay objects, it is easy to see that for $\R$-constructible enhanced ind-sheaves the associated moderate and rapid decay objects are related by duality.

Let $X$ be a complex manifold and $f\colon X\to \C$ a holomorphic function.
\begin{prop}[Local duality pairing on $\Xf$]\label{prop:duality1} Let $K\in\EbRcIk{X}$. Then, there is a canonical isomorphism
$$\mathrm{D}_{\Xf} K^{\mathrm{mod}\,f}\iso (\DE K)^{\mathrm{rd}\,f}.$$
That is, there exists a perfect pairing 
$$
K^{\mathrm{mod}\,f} \otimes_{k_{\Xf}} (\DE K)^{\mathrm{rd}\,f} \longrightarrow \omega_{\Xf},
$$
where $\omega_{\Xf}$ is the Verdier dualizing complex in $\Dbk{\Xf}$.
\end{prop}
\begin{proof}
	There is the chain of isomorphisms
	\begin{align*}
		\mathrm{D}_{\Xf} \shbuf \big( \EE \fj_*\EE j^{-1} K\big)&\iso \shbuf \big( \mathrm{D}_{\Xf}^\mathrm{E} \EE \fj_*\EE \tilde{j}_f^! \EE \varpi^! K\big)\\
		&\iso \shbuf\big(\mathrm{D}_{\Xf}^\mathrm{E} \RIhom(\pi^{-1}k_{X^*}, \EE \varpi^! K)\big)\\
		&\iso \shbuf\big( \pi^{-1}k_{X^*}\otimes \mathrm{D}_{\Xf}^\mathrm{E} \EE \varpi^! K\big)\\
		&\iso \shbuf\big( \EE j_{!!}^f\EE {j^f}^{-1} \EE \varpi^{-1} \DE K\big),
	\end{align*}
where the first isomorphism follows from the commutation of sheafification with duality (see \cite[Lemma 3.13]{DK21sh}), the third isomorphism follows from \cite[Lemma 4.3.2, Proposition 4.9.13]{DK16}, and in the last isomorphism we have used the fact that duality interchanges inverse image and exceptional inverse image (see \cite[Proposition 3.3.3]{DK19}). Moreover, in the second and fourth isomorphism we have applied \cite[Lemma 2.7.6]{DK19}. 

Note that, except for the second line, all the isomorphisms require $\R$-constructi\-bil\-i\-ty, so it is essential here that $K$ is $\R$-constructible and that direct and inverse images preserve this $\R$-constructibility (see \cite[Proposition 3.3.3]{DK19}).
\end{proof}

By restricting the moderate growth and rapid decay functors to the boundary of the real blow-up, we immediately obtain as a consequence a duality between the functors $\psi_f^{\leq 0}$ and $\psi_f^{\geq 0}$ (see Definition~\ref{def:Cycles}), up to a shift in the derived category $\mathrm{D}_{\R\text{-}c}^\mathrm{b}(k_{\partial \Xf})$.

\begin{cor}[Local duality pairing on $\partial \Xf$]\label{cor:duality2}
For $K \in \EbRcIk{X}$, there is a canonical isomorphism
\begin{align*}
\mathrm{D}_{\partial \Xf} (\psi_f^{\leq 0} K) \iso \psi_f^{\geq 0} (\DE K) [1],
\end{align*}
so that the associated pairing
$$
\psi_f^{\leq 0}(K) \otimes_{k_{\Xf}} \psi_f^{\geq 0}(\DE K)[1] \longrightarrow \omega_{\partial \Xf}
$$
is perfect.
\end{cor}

\begin{proof}
This is a similarly straightforward computation, using in addition the fact that Verdier dualizing interchanges $i^{-1}$ and $i^!$.
\end{proof}

\subsection{The global duality pairings}\label{subsec:dualitypairing}
Applying (derived) global sections to the objects in Proposition \ref{prop:duality1} yields further duality statements that we can explore.

Let $K\in\EbRcIk{X}$ be an $\R$-construcible enhanced ind-sheaf.

Specifically, the isomorphism $K^{\mathrm{mod}\,f} \iso D_{\Xf}(\DE K)^{\mathrm{rd}\,f}$ implies the isomorphism 
\begin{align}\label{prop:pairing2}
\RGamma(\Xf;K^{\mathrm{mod}\,f}) &\xrightarrow{\thicksim} \Hom(\RGamma_c(\Xf;(\DE K)^{\mathrm{rd}\,f}),k) \\
&=\vcentcolon \RGamma_c\big (\Xf;(\DE K)^{\mathrm{rd}\,f} \big )^\vee\notag
\end{align}
of bounded complexes of finite-dimensional $k$-vector spaces. Likewise, 
$$
\mathrm{D}_{\Xf} \big ( K^{\mathrm{mod}\,f}\big ) \iso (\DE K)^{\mathrm{rd}\,f}
$$
implies 
\begin{align*}
\RGamma(\Xf;(\DE K)^{\mathrm{rd}\,f}) \xrightarrow{\thicksim} \RGamma_c(\Xf; K^{\mathrm{mod}\,f})^\vee . 
\end{align*}

Either of these isomorphisms yields the following result.

\begin{cor}\label{prop:duality3}
The pairing 
$$
\hyp^\ell(\Xf;K^{\mathrm{mod}\,f}) \otimes_k \hyp_c^{-\ell}(\Xf;(\DE K)^{\mathrm{rd}\,f}) \longrightarrow k
$$
is perfect for all $K \in \EbRcIk{X}$ and $\ell \in \Z$.
\end{cor}

\begin{proof}
This follows directly from Proposition \ref{prop:pairingref} and the above local duality, Proposition \ref{prop:duality1}.
\end{proof}

Likewise, Proposition \ref{prop:pairingref} and Corollary \ref{cor:duality2} imply the following result.

\begin{cor}\label{cor:pairing2}
The pairing
$$
\hyp^\ell(\partial \Xf;\psi_f^{\leq 0}K) \otimes_k \hyp_c^{-\ell+1}(\partial \Xf;\psi_f^{\geq 0}\,\DE K) \longrightarrow  k
$$
is perfect for all $K \in \EbRcIk{X}$ and $\ell \in \Z$. 
\end{cor}

\bigskip

\section{Duality between de Rham functors with growth conditions}\label{sec:TWduality}

\subsection{A duality of Kashiwara--Schapira}
In this section, we briefly recall a duality statement between tempered and Whitney holomorphic functions proved by M.\ Kashiwara and P.\ Schapira in \cite{KS96} and \cite{KS16}.

It is well-known that the space of distributions is the topological dual of the space of compactly supported smooth functions (often called \emph{test functions}). It is important to note that these spaces are infinite-dimensional complex vector spaces and this duality is really a duality of \emph{topological} vector spaces, and not a duality in the category of complex vector spaces. The authors of \cite{KS96} use the language of Fréchet nuclear spaces (vector spaces of type $\FN$) and duals of Fréchet nuclear spaces (vector spaces of type $\DFN$) to formulate their results (see \cite{Gro55} for these notions, as well as \cite{RR74} and the references therein). Our proof strategy is similar in character to the classical isomorphism between de Rham and solution complexes for holonomic $\cD$-modules (see Propositions 4.6.4 and 4.7.9 of \cite{HTT08}, where one uses the duality between Frech\'et--Schwartz (FS) and duals of Frech\'et--Schwartz (DFS) topological vector spaces).

The notions of tempered distributions and Whitney functions were studied in \cite{KS96} and they are closely related to the notions of moderate growth and rapid decay (cf., e.g., \cite[Proposition~7.2.10, Lemma~5.1.4]{DK16}). On a real analytic manifold $X$, recall the functor $\bullet\wtens\Cinf{X}\colon \ModRcC{X}\to \ModD{X}$ from loc.~cit.\ as well as the ind-sheaf of tempered distributions $\Dbt{X}$ from \cite{KS01}. The ind-sheaf $\Dbtv{X}$ is the ind-sheaf of tempered distribution densities on $X$ (see, e.g., \cite{KS16}). In \cite{KS96}, the authors proved the following duality result for these spaces (formulated here in the more modern notation of \cite{KS01} and \cite{KS16}).
\begin{prop}[{\cite[Proposition 2.2]{KS96}}]
Let $X$ be a real analytic manifold and $F\in\ModRcC{X}$. Then there exist natural topologies of type $\FN$ on $\Gamma(X;F\wtens\Cinf{X})$ and of type $\DFN$ on $\Gamma_{\mathrm{c}}(X;\alpha_X\RIhom(F,\Dbtv{X}))$ and they are dual to each other.
\end{prop}

This result is then used to deduce the following duality in the context of holomorphic solutions of a coherent $\cD_X$-module on a complex manifold.
We note that the transition from $\Cinf{X}$ (resp.\ $\Dbtv{X}$) to $\cO_X$ (resp.\ $\Omega_X$) amounts to taking Dolbeault complexes with coefficients in those objects (see \cite[§5]{KS96}).

\begin{thm}[{\cite[Theorem 6.1]{KS96}}]
Let $X$ be a complex manifold of dimension $d_X$, $\cM\in \mathrm{D}^\mathrm{b}_{\mathrm{coh}}(\cD_X)$ and $F,G\in\DbRcC{X}$. Then the spaces $\RR \Gamma(X;\RR \cH om_{\cD_X}(\cM\otimes G,F\wtens\cO_X))$ and $\RR\Gamma_{\mathrm{c}}(X; \alpha_X\RIhom(F,\Omega^\mathrm{t}_X)[d_X]\Ltens_{\cD_X}(\cM\otimes G))$ are objects of $\mathrm{D}^\mathrm{b}(\FN)$ and $\mathrm{D}^\mathrm{b}(\DFN)$, respectively, and are dual to each other, functorially in $\cM$, $F$ and $G$.
\end{thm}

Later, in \cite{KS16}, this global result was extended to the following local statement in the category of $\R$-constructible sheaves.
\begin{thm}[{see \cite[Theorem 2.5.13]{KS16}}]
Let $X$ be a complex manifold of dimension $d_X$, $\cM\in \mathrm{D}^\mathrm{b}_{\mathrm{hol}}(\cD_X)$ and $F,G\in\DbRcC{X}$. Then the two objects $\RR \cH om_{\cD_X}(\cM,F\wtens\cO_X)$ and $\alpha_X\RIhom(F,\Omega^\mathrm{t}_X)[d_X]\Ltens_{\cD_X}\cM$ are dual to each other in the category $\DbRcC{X}$.
\end{thm}
Let us remark that the key in deriving this result was the proof of the $\R$-constructibility of these two objects, which was derived from the constructiblity of enhanced solutions established in \cite{DK16}. In particular, this allows the authors of \cite{KS16} to forget the topology, so this duality is really a duality of sheaves of complex vector spaces.

\subsection{A conjecture of Sabbah}\label{subsec:SabbahProof}

In this subsection, we will prove the following theorem, which is Sabbah's Conjecture, as discussed in Section~\ref{sec:intro} (Theorem~\ref{conj:SabDuality}).
\begin{thm}[{cf.\ \cite[Conjecture 4.13]{Sab21}}]\label{thm:SabDuality}
Let $X$ be a complex manifold, $f\colon X\to \C$ a holomorphic function, and $\cM$ a holonomic $\cD_X$-module. Then there is a natural isomorphism 
$$\mathrm{D}_{\Xf}\DRmodf(\cM)\iso \DRrdf(\mathbb{D}_X\cM)$$
in $\DbRcC{\Xf}$.
\end{thm}

In the rest of this subsection, we will first give a proof of this conjecture in the case of a normal crossing divisor using arguments very similar to the ones of Kashiwara--Schapira \cite{KS96}, before proving the general version.

We first show the following variant on the real blow-up of $X$ along a normal crossing divisor.
\begin{prop}\label{prop:conjSabNCD}
Let $D\subset X$ be a simple normal crossing divisor. Let $\cM$ be a holonomic $\cD_X$-module. Then there is an isomorphism in $\DbC{\XD}$
$$\mathrm{D}_{\XD}\DRmodD(\cM)\iso \DRrdD(\mathbb{D}_X\cM).$$
\end{prop}

Let us first rewrite the objects in the statement of Proposition~\ref{prop:conjSabNCD} in terms of (holomorphic) tempered distributions and Whitney functions, as studied in \cite{KS96}, for example. 

In the following, we will just write $\wt{X}$, $\varpi$ for $\XD$, $\varpi_D$. Let $\Xtot$ be the total real blow-up of $X$ along a normal crossing divisor, and let $\widetilde{X}\subset \Xtot$ be the real blow-up (see Section~\ref{subsec:realblowup} for these constructions).

For the sheaf of moderate growth holomorphic functions, we have
\begin{align*}
    \Amod &\iso \alpha_{\wt{X}} \Ot{\wt{X}} \iso \alpha_{\wt{X}} \varpi^! \RIhom(\C_{X^*},\cO^\mathrm{t}_{X})
\end{align*}
and hence the moderate growth de Rham complex of a holonomic $\cD_X$-module $\cM$ is
\begin{align}\label{eq:mod}
\DRmod(\cM)&\iso (\varpi^{-1}\Omega_X \otimes_{\varpi^{-1}\cO_X} \Amod) \Ltens_{\varpi^{-1}\cD_X} \varpi^{-1} \cM\notag \\
&\iso \alpha_{\widetilde{X}} \varpi^! \RIhom(\C_{X^*},\Omega^\mathrm{t}_{X}) \Ltens_{\varpi^{-1}\cD_X} \varpi^{-1}\cM\\
&\iso \alpha_{\wt{X}} \Omega^\mathrm{t}_{\wt{X}} \Ltens_{\varpi^{-1}\cD_X} \varpi^{-1}\cM\notag
\end{align}

On the other hand, the sheaf of smooth functions on $\wt{X}$ with rapid decay at $\partial\wt{X}$ is
$$\mathscr{C}^{\infty,\mathrm{rd}}_{\widetilde{X}} \vcentcolon= (\C_{X^*}\overset{\mathrm{w}}{\otimes} \mathscr{C}^\infty_{\Xtot})\big|_{\widetilde{X}}$$
and the sheaf of holomorphic function with rapid decay at the boundary is then
$$\Ard \vcentcolon= \RR \cH om_{\varpi^{-1}\cD_{\overline{X}}}(\varpi^{-1}\cO_{\overline{X}}, \mathscr{C}^{\infty,\mathrm{rd}}_{\widetilde{X}}),$$
and the rapid decay de Rham complex of the dual of $\cM$ is
\begin{align}\label{eq:rd}
	\DRrd(\mathbb{D}_X\cM) &\iso (\varpi^{-1}\Omega_X \otimes_{\varpi^{-1}\cO_X} \Ard) \Ltens_{\varpi^{-1}\cD_X} \varpi^{-1} \mathbb{D}_X\cM\notag \\
	&\iso \RR \cH om_{\varpi^{-1}\cD_X}(\varpi^{-1} \cM, \Ard)[d_X]
\end{align}
\\
In order to prove Proposition~\ref{prop:conjSabNCD}, the aim is to get a duality between \eqref{eq:mod} and \eqref{eq:rd}, and this reminds us of the duality in \cite[(2.5.12), Theorem 2.5.13]{KS16}. However, our statement is a duality on the real blow-up and it does not follow directly from loc.~cit., but our proof will proceed along the same lines.

We will first construct a pairing between the two objects \eqref{eq:mod} and \eqref{eq:rd}.

\subsubsection*{A pairing on $\Xtot$}
Since $\Xtot$ is a real analytic manifold, we get from \cite[(2.5.11)]{KS16} a pairing
\begin{equation}\label{eq:pairingXtot}
	(\C_{X^*}\wtens \Cinf{\Xtot}) \otimes \alpha_{\Xtot}\RIhom(\C_{X^*},\Dbtv{\Xtot}) \to \omega_{\Xtot}.
\end{equation}

\subsubsection*{The induced pairing on $\widetilde{X}$}
Equation \eqref{eq:pairingXtot} is equivalent to a morphism
$$ \alpha_{\Xtot}\RIhom(\C_{X^*},\Dbtv{\Xtot}) \to \mathrm{D}_{\Xtot} (\C_{X^*}\wtens \Cinf{\Xtot}). $$
Applying the exceptional inverse image along the (closed) embedding $\itot\colon \widetilde{X}\hookrightarrow\Xtot$, we obtain
\begin{equation}\label{eq:dualityMorphBUtot}
	(\itot)^! \alpha_{\Xtot}\RIhom(\C_{X^*},\Dbtv{\Xtot}) \to  \mathrm{D}_{\widetilde{X}} (\itot)^{-1} (\C_{X^*}\wtens \Cinf{\Xtot}).
\end{equation}

Now observe that the left-hand side can be manipulated as follows:
\begin{align}\label{eq:alphaDbtemp}
	(\itot)^! \alpha_{\Xtot}\RIhom(\C_{X^*}&,\Dbtv{\Xtot}) \notag\\
	&\iso (\itot)^{-1} \RR \cH om\big(\C_{\widetilde{X}},\alpha_{\Xtot}\RIhom(\C_{X^*},\Dbtv{\Xtot})\big)\notag \\
	&\iso (\itot)^{-1} \RR \cH om\big(\beta_{\Xtot}\C_{\widetilde{X}},\RIhom(\C_{X^*},\Dbtv{\Xtot})\big)\notag \\
	&\iso  (\itot)^{-1} \alpha_{\Xtot} \RIhom(\beta_{\Xtot}\C_{\widetilde{X}}\otimes \iota_{\Xtot}\C_{X^*},\Dbtv{\Xtot})\notag \\
	&\iso  \alpha_{\widetilde{X}} (\itot)^{-1} \RIhom(\C_{X^*},\Dbtv{\Xtot})
\end{align}
Here, the second isomorphism follows from \cite[Proposition 5.1.10]{KS01}, and the last one from the fact that $\beta_{\Xtot}\C_{\widetilde{X}}\otimes \iota_{\Xtot}\C_{X^*}\iso \iota_{\Xtot}\C_{X^*}$ (see \cite{KS01} and also \cite[Proof of Lemma 3.2]{Ho22}).

Next, let us prove the following result similar to \cite[Lemma 2.5.7]{KS16}.
\begin{lemma}\label{lemma:DbtPullback}
	Set $Y\vcentcolon=(\varpitot)^{-1}(X^*)$. Then there is an isomorphism
	\begin{equation}\label{eq:shriekDb}
		\RIhom(\C_Y,\Dbtv{\Xtot})\iso (\varpitot)^! \RIhom(\C_{X^*}, \Dbtv{X}).
	\end{equation}
\end{lemma}
\begin{proof}
	Note that, locally, $Y$ is the disjoint union of $2^r$ connected components ($r$ being the number of smooth components of the normal crossing divisor $D$), each homeo\-morphically mapped to $X^*$ via $\varpitot$. Let us write $Y=\bigsqcup_{\nu\in\{+,-\}^r} X^*\times\{\nu\}$. (Recall that $X^*\times \{+,\ldots,+\}$ is the component canonically identified with $X^*$ in $\widetilde{X}\subset \Xtot$.)
	Then we have
	\begin{align*}
		(\varpitot)^!\RIhom(\C_{X^*},\Dbtv{X}) &\iso \RIhom\big((\varpitot)^{-1}\C_{X^*},(\varpitot)^!\Dbtv{X}\big)\\
		&\iso \bigoplus_{\nu\in\{+,-\}^r} \RIhom\big(\C_{X^*\times\{\nu\}},(\varpitot)^!\Dbtv{X}\big).
	\end{align*}
	Now, by \cite[Theorem 2.5.6]{KS16}, one has $(\varpitot)^!\Dbtv{X}\iso \Dbtv{\Xtot}\Ltens_{\cD_{\Xtot}}\cD_{\Xtot\to X}$, and since $\cD_{\Xtot}\to\cD_{\Xtot\to X}$ is an isomorphism on each $X^*\times\{\nu\}$, we obtain the desired result.
\end{proof}
Applying the functor $(\itot)^!$ to the isomorphism \eqref{eq:shriekDb}, the left-hand side is the following:
\begin{align*}
	(\itot)^! \RIhom(\C_Y,\Dbtv{\Xtot}) &\iso (\itot)^{-1}\RIhom(\C_{\widetilde{X}},\RIhom(\C_Y,\Dbtv{\Xtot}))\\
	&\iso (\itot)^{-1}\RIhom(\C_{\widetilde{X}}\otimes \C_Y,\Dbtv{\Xtot})\\
	&\iso (\itot)^{-1}\RIhom(\C_{X^*},\Dbtv{\Xtot})
\end{align*}
Hence, taking together the isomorphism \eqref{eq:shriekDb} with our computation \eqref{eq:alphaDbtemp}, we obtain
\begin{align}\label{eq:AmodManipulation}
    (\itot)^! \alpha_{\Xtot}\RIhom(\C_{X^*},\Dbtv{\Xtot}) &\iso \alpha_{\widetilde{X}} (\itot)^! (\varpitot)^! \RIhom(\C_{X^*}, \Dbtv{X}) \\
    &\iso \alpha_{\widetilde{X}} \varpi^! \RIhom(\C_{X^*}, \Dbtv{X})\notag
\end{align}
and hence the morphism \eqref{eq:dualityMorphBUtot} yields a pairing
\begin{equation}\label{eq:pairingBU}
(\itot)^{-1} (\C_{X^*}\wtens \Cinf{\Xtot}) \otimes \alpha_{\widetilde{X}} \varpi^! \RIhom(\C_{X^*}, \Dbtv{X}) \to \omega_{\widetilde{X}}.
\end{equation}

\subsubsection*{The pairing between de Rham functors}
Making the transition to Dolbeault complexes (i.e., going to the holomorphic setting) and adding a $\cD$-module, the above pairing will induce a pairing
\begin{equation}\label{eq:pairingDR}
\DRrd(\mathbb{D}_X\cM) \otimes \DRmod(\cM) \to \omega_{\widetilde{X}}
\end{equation}
(just as (2.5.11) induces (2.5.12) in \cite{KS16}).

\subsubsection*{Perfectness of the pairing}
We first prove an analogue (indeed, a corollary) of \cite[Proposition 2.2]{KS96}, which will be the key to the proof of the perfectness.
\begin{prop}\label{prop:dualitySmoothDist}
	There exist natural topologies of type $\FN$ on $\Gamma(\wt{X};(\itot)^{-1}(\C_{X^*}\wtens\Cinf{\widetilde{X}}))$ and of type $\DFN$ on $\Gamma_c(\wt{X};(\itot)^{-1}\RR \cH om(\C_{X^*},\Dbtv{\Xtot}))$ and the two spaces are dual to each other with respect to these topologies.
\end{prop}
\begin{proof}
	Since $\Xtot$ is a real analytic manifold, \cite[Proposition 2.2]{KS96} gives us a (topological) duality between the $\FN$ space $\Gamma(\Xtot;\C_{X^*}\wtens\Cinf{\Xtot})$ and the $\DFN$ space $\Gamma_c(\Xtot;\RIhom(\C_{X^*},\Dbtv{\Xtot}))$. To conclude, we observe that both $\C_{X^*}\wtens\Cinf{\Xtot}$ and $\RIhom(\C_{X^*},\Dbtv{\Xtot})$ are supported on the closed subspace $\widetilde{X}\subset \Xtot$, and hence
	$$\Gamma\big(\wt{X};(\itot)^{-1}(\C_{X^*}\wtens\Cinf{\Xtot})\big) \iso \Gamma\big(\Xtot;\C_{X^*}\wtens\Cinf{\Xtot}\big)$$
	and
	$$\Gamma_c\big(\wt{X};(\itot)^{-1}\RR \cH om(\C_{X^*},\Dbtv{\Xtot})\big)\iso \Gamma_c\big(\Xtot;\RIhom(\C_{X^*},\Dbtv{\Xtot})\big).$$
\end{proof}

This implies (going to the holomorphic setting similarly to \cite[Proposition 5.2]{KS96}) a duality between the objects
$$\RR \Gamma(\wt X;\Ard)\in \mathrm{D}^\mathrm{b}(\FN)$$
and
$$\RR \Gamma_c(\wt X;\alpha_{\widetilde{X}}\Omega^\mathrm{t}_{\widetilde{X}}[d_X])\in \mathrm{D}^\mathrm{b}(\DFN).$$

We can then use the technique of \cite[Theorem 6.1 and Appendix]{KS96} to prove the following.
\begin{prop}\label{prop:topdualityBU}
	Let $\cM\in\DbholD{X}$ and let $G\in\DbRcC{\widetilde{X}}$. Then we can define
	$$\RR \Gamma\big(\wt X;\RR\cH om_{\varpi^{-1}\cD_X}(\varpi^{-1}\cM\otimes G,\Ard)\big)\iso \RR \mathrm{Hom}\big(G,\DRrd(\mathbb{D}_X\cM)\big)\in \mathrm{D}^\mathrm{b}(\FN)$$
	and
	$$\RR\Gamma_c\big(\wt X;\DRmod(\cM)\otimes G\big)\in \mathrm{D}^\mathrm{b}(\DFN)$$
	and these are dual to each other.
\end{prop}
\begin{proof}
	By the methods developed in loc.~cit. (whose details we will not elaborate on here, see the proof of \cite[Theorem 6.1]{KS96}), the proof can be reduced to the case $\cM=(\cD_X)_{U}$ for some relatively compact open $U\subset X$ and $G=\C_V$ for some relatively compact open subanalytic $V\subset\wt{X}$, and this reduces the statement to Proposition~\ref{prop:dualitySmoothDist} and its holomorphic analogue.
\end{proof}

The final step in proving Proposition~\ref{prop:conjSabNCD} is now the following lemma, which is analogous to the idea of \cite[Lemma 2.5.12]{KS16}.

\begin{lemma}\label{lemma:constrDual}
	If one of the objects $\DRrd(\mathbb{D}_X\cM)$ or $\DRmod(\cM)$ is $\mathbb{R}$-constructible, then the pairing \eqref{eq:pairingDR} is perfect and the objects $\DRrd(\mathbb{D}_X\cM)$ and $\DRmod(\cM)$ are dual to each other in the category $\DbRcC{\widetilde{X}}$. In particular, the other object is also $\R$-constructible.
\end{lemma}
\begin{proof}
	Assume that $\DRmod(\cM)$ is $\R$-constructible.
	Apply the functor $\RR \Gamma_c(U;\bullet)$ for any relatively compact open subanalytic $U\subset \widetilde{X}$ to the morphism
	$$\DRmod(\cM)\to\mathrm{D}_{\widetilde{X}} \DRrd(\mathbb{D}_X\cM)$$
	coming from the pairing \eqref{eq:pairingDR} to obtain
	$$\RR\Gamma(U;\DRrd(\mathbb{D}_X\cM)) \to \RR\Gamma_c(U;\DRmod(\cM))^*.$$
	Since $\DRmod(\cM)$ was assumed to be $\mathbb{R}$-constructible, we know (see \cite[Corollary 8.4.11]{KS90}) that $\RR\Gamma_c(U;\DRmod(\cM))$ is a perfect complex and therefore has, in particular, finite-dimensional cohomologies, and hence the topological dual and the usual dual as a vector space coincide.
	This morphism is an isomorphism by Proposition~\ref{prop:topdualityBU}, which concludes the proof.
	
	The case where $\DRrd(\cM)$ is $\R$-constructible is similar.
\end{proof}

We can now finally prove the above duality statement in the case of a normal crossing divisor.

\begin{proof}[Proof of Proposition~\ref{prop:conjSabNCD}]
By Lemma~\ref{lemma:constrDual}, we are reduced to proving that the complex $\DRmod(\cM)$ is $\R$-constructible. This is, however, clear by (a version on $\XD$ of) Lemma~\ref{lemma:constructible} since we have already proved $\DRmodD(\cM)\iso \shbuD\big(\EE \Dj_*\EE j^{-1}\DRE(\cM)\big)$ (see Proposition~\ref{prop:DRmod}) and we know that $\DRE(\cM)$ is $\R$-constructible (see \cite[Theorem 9.3.2]{DK16}). This concludes the proof.
\end{proof}

The assertion of Theorem~\ref{thm:SabDuality} in the case where $f^{-1}(0)$ is a simple normal crossing divisor now follows as an easy corollary.

\begin{cor}\label{cor:dualityIsoXfNCD}
Let $f\colon X\to \C$ be a holomorphic function such that $D\defeq f^{-1}(0)$ is a simple normal crossing divisor. Let $\cM$ be a holonomic $\cD_X$-module. Then there is an isomorphism in $\DbC{\Xf}$
$$\mathrm{D}_{\Xf}\DRmodf(\cM)\iso \DRrdf(\mathbb{D}_X\cM).$$
\end{cor}
\begin{proof}
This follows directly from Proposition~\ref{prop:conjSabNCD} by applying the direct image functor along the natural morphism $\varpi_{D,f}\colon \XD\to\Xf$. We note that this morphism is proper and hence self-dual, and that this direct image turns the moderate growth (resp. rapid decay) de Rham complex on $\XD$ into the moderate growth (resp.\ rapid decay) de Rham complex on $\Xf$ (see \cite[Proposition 4.7.4]{Moc14}).
\end{proof}

Let us now consider the case of a complex manifold $X$ and a holomorphic function $f\colon X\to \C$ where $f^{-1}(0)$ does not necessarily define a simple normal crossing divisor. The techniques of resolution of singularities allow us to reduce the general case to the one proved above, once we have the following lemma.
\begin{lemma}\label{conj:natMorph}
	For any complex manifold $X$, any holomorphic function $f\colon X\to \C$ and any object $\cM\in \DbholD{X}$, there exists a natural morphism
	\begin{equation}\label{eq:natMorphXf}
		\DRrdf(\D_X\cM)\to \mathrm{D}_{\widetilde{X}_f}\DRmodf(\cM),
	\end{equation}
functorial in $\cM$.
\end{lemma}

\begin{proof}
First, note that we have an isomorphism
	\begin{equation}\label{DRrd}
		\DRrdf(\D_X\cM) \iso \sHom_{\varpi_f^{-1}\cD_X}(\varpi_f^{-1}\cM,\Ardf)[d_X]
	\end{equation}
for any coherent $\cD_X$-module $\cM$. (This is proved by standard arguments, similarly to \cite[Lemma~5.5.1]{Moc14}, for instance.) Here, $d_X$ denotes the (complex) dimension of $X$.

Note moreover that we have a natural ($\varpi_f^{-1}\cD_X$-linear) morphism
\begin{equation}\label{mult}
\Amodf \otimes_{\varpi_f^{-1}\cO_X} \Ardf\to \Ardf
\end{equation}
given by multiplication of functions. (Recall that $\Amodf$ and $\Ardf$ are flat over $\varpi_f^{-1}\cO_X$.)

With this in hand, we construct the desired morphism as follows:
	\begin{align*}
		\DRrdf(\D_X\cM) &\overset{\eqref{DRrd}}{\iso} \sHom_{\varpi_f^{-1}\cD_X}(\varpi_f^{-1}\cM,\Ardf)[d_X]\\
		&\to \sHom_{\C_{\Xf}}\big(\DRmodf(\cM),\varpi_f^{-1}\Omega_X\Ltens_{\varpi_f^{-1}\cD_X} (\Amodf \otimes_{\varpi_f^{-1}\cO_X}\Ardf)\big)[d_X]\\
		&\overset{\eqref{mult}}{\to} \sHom_{\C_{\Xf}}(\DRmodf(\cM),\varpi_f^{-1}\Omega_X\Ltens_{\varpi_f^{-1}\cD_X} \Ardf)[d_X]\\
		&\iso \sHom_{\C_{\Xf}}\big(\DRmodf(\cM),\DRrdf(\cO_X)\big)[d_X]\\
		&\iso \sHom_{\C_{\Xf}}\big(\DRmodf(\cM),\widetilde{j}_!\C_{X^*}[d_X]\big)[d_X]\\
		&\to \sHom_{\C_{\Xf}}\big(\DRmodf(\cM),\omega_{\Xf}\big) = \mathrm{D}_{\Xf}\DRmodf(\cM).
	\end{align*}

The morphism in the second line follows by applying the functor $$(\varpi_f^{-1}\Omega_X\otimes_{\varpi_f^{-1}\cO_X} \Amodf)\Ltens_{\varpi_f^{-1}\cD_X}(\bullet)$$ to both components.

The isomorphism in the fifth line follows since
$$\DRrdf(\cO_X)\iso \widetilde{j}_!\C_{X^*}[d_X],$$
where $\widetilde{j}\colon X^*\hookrightarrow \Xf$ is the embedding (see, for example, \cite[Example 4.12]{Sab21}).

The last isomorphism follows from the natural morphism $\widetilde{j}_!\C_{X^*}[2d_X]\to \omega_{\Xf}$, which comes as the adjoint to the identity $\C_{X^*}[2d] \iso \omega_{X^*} \iso \widetilde{j}^!\omega_{\Xf}$.
\end{proof}

We will now show that this morphism is in fact an \emph{isomorphism}, yielding the following statement.
\begin{prop}\label{prop:conjNotNCD}
	For any complex manifold $X$, any holomorphic function $f\colon X\to \C$ and any $\cM\in\DbholD{X}$, the morphism \eqref{eq:natMorphXf} is an isomorphism.

 In particular, we obtain an isomorphism $$\DRrdf(\cM)\iso \shbuf\big(\EE \fj_{!!}\EE j^{-1}\DRE(\cM)\big).$$
 
\end{prop}
\begin{proof}
	First, note that if $\cM'\To\cM\To\cM''\ToPO$ is a distinguished triangle in $\DbholD{X}$, then this induces a morphism of distinguished triangles
	$$\begin{tikzcd}
		\DRrdf(\D_X\cM') \arrow{r}\arrow{d} & \DRrdf(\D_X\cM) \arrow{r}\arrow{d} & \DRrdf(\D_X\cM'') \arrow{r}{+1} \arrow{d} & \text{}\\
		\mathrm{D}_{\widetilde{X}_f}\DRmodf(\cM') \arrow{r} & \mathrm{D}_{\widetilde{X}_f}\DRmodf(\cM) \arrow{r} & \mathrm{D}_{\widetilde{X}_f}\DRmodf(\cM'') \arrow{r}{+1} & \text{}
	\end{tikzcd}$$
	and, by the axioms of a triangulated category, if two of the vertical arrows are isomorphisms, so is the third. Hence, by a standard induction argument on the amplitude of $\cM$ (see, e.g., \cite[Proof of Lemma 7.3.7]{DK16}), we can reduce to the case when $\cM$ is concentrated in degree $0$.
	
	Moreover, if $U\subset X$ is open, then $f$ defines a holomorphic function on $U$, and we denote by $\widetilde{U}_f\subset \widetilde{X}_f$ the corresponding real blow-up space. One has $$\mathrm{DR}^\mathrm{rd}_{\widetilde{U}_f}(\mathbb{D}_U \cM|_U) \iso \DRrdf(\mathbb{D}_X\cM)|_{\widetilde{U}_f}$$ and $$\mathrm{D}_{\widetilde{U}_f}\mathrm{DR}^\mathrm{mod}_{\widetilde{U}_f}(\cM|_U) \iso (\mathrm{D}_{\widetilde{X}}\DRmodf(\cM))|_{\widetilde{U}_f}.$$
	Therefore, \eqref{eq:natMorphXf} is an isomorphism for $\cM$ if it is one for all $\cM|_{U_i}$ on a covering $X=\bigcup_{i\in I} U_i$; in other words, the statement is local. We can therefore assume that there exist a complex manifold $Y$ and a projective morphism $\tau\colon Y\to X$ such that, if we set $g\vcentcolon= f\circ \tau$, $E\vcentcolon= g^{-1}(0) = \tau^{-1}(f^{-1}(0))\subset Y$ is a simple normal crossing divisor.
	
	We know from Corollary~\ref{cor:dualityIsoXfNCD} that
	$$\mathrm{DR}^\mathrm{rd}_{\widetilde{Y}_g}(\D_Y \mathrm{D}\tau^*\cM)\to \mathrm{D}_{\wt{Y}_g}\mathrm{DR}^\mathrm{mod}_{\widetilde{Y}_g}(\mathrm{D}\tau^*\cM)$$
	is an isomorphism. Applying the functor $\mathrm{D}\tau_*$ to this isomorphism, we obtain (by \cite[Corollary 4.7.3]{Moc14} and the well-known fact that direct images for constructible sheaves and holonomic $\cD$-modules along proper maps commute with duality)
	$$\DRrdf(\D_X\mathrm{D}\tau_*\mathrm{D}\tau^*\cM)\overset{\sim}{\To} \mathrm{D}_{\Xf}\DRrdf(\mathrm{D}\tau_*\mathrm{D}\tau^*\cM).$$
	
	Now it remains to note that $\cM$ is a direct summand of $\mathrm{D}\tau_*\mathrm{D}\tau^*\cM$ and since the morphism \eqref{eq:natMorphXf} is functorial (in particular, compatible with direct sums), we deduce the isomorphism
	$$\DRrdf(\D_X\cM)\overset{\sim}{\To} \mathrm{D}_{\Xf}\DRmodf(\cM)$$
    as desired.

    The second part of the proposition follows easily (as in Corollary~\ref{cor:DRrdXD}).
\end{proof}

\begin{rem}
	Note that the proof just given might appear similar to the method given in \cite[Lemma 7.3.7]{DK16}, where the authors show how to reduce a proof of a statement involving a complex manifold and a holonomic $\cD$-module to the case of a $\cD$-module with normal form on a manifold with a normal crossing divisor, using the classificaion results due to C.\ Sabbah \cite{Sab00}, K.\ Kedlaya \cite{Ked10,Ked11} and T.\ Mochizuki \cite{Moc09,Moc11}. However, here we do not use any reduction on the form of our $\cD$-module (such as a normal form or a good filtration). Indeed, the reduction we perform is only to the case where the divisor has normal crossings (and the $\cD$-module is still arbitrary holonomic).
 \end{rem}

 We have thus completed the proof of Theorem~\ref{thm:SabDuality}.

 With this result in hand, we can now derive the statement about the expression for the moderate de Rham complex in terms of the enhanced de Rham complex, which has been motivated in \eqref{eq:ideaDRrd}, showing that the expression we were expecting for the rapid decay de Rham complex in terms of the enhanced de Rham complex is indeed true.

\begin{cor}\label{cor:DRrdXD}
If $D\subset X$ is a simple normal crossing divisor, we have an isomorphism
$$\DRrdD(\cM)\iso \shbuD\big(\EE \Dj_{!!}\EE j^{-1}\DRE(\cM)\big).$$
If $f\colon X\to \C$ is a holomorphic function, we have an isomorphism
$$\DRrdf(\cM)\iso \shbuf\big(\EE \fj_{!!}\EE j^{-1}\DRE(\cM)\big).$$
\end{cor}
\begin{proof}
The proofs of the two statements are identical, so let us only give the argument for $\Xf$.

Recall (from Proposition~\ref{prop:DRmodf}) that
$$\DRmodf(\cM)\iso \shbuf(\EE \fj_*\EE j^{-1}\DRE(\cM)).$$
Then the assertion of the corollary is just the dual of this isomorphism: The left-hand sides are dual to each other by Theorem~\ref{thm:SabDuality} just proved. For the right-hand sides, we see that
\begin{align*}
    \shbuf(\EE \fj_*\EE j^{-1}\DRE(\mathbb{D}_X\cM))&\iso \shbuf(\EE \fj_*\EE j^{-1}\mathrm{D}^\EE_{X}\DRE(\cM))\\
    &\iso \shbuf(\mathrm{D}^\EE_{\Xf}\EE \fj_{!!}\EE j^!\DRE(\cM))\\
    &\iso \mathrm{D}_{\Xf}\shbuf(\EE \fj_{!!}\EE j^{-1}\DRE(\cM))
\end{align*}
since duality commutes with the enhanced de Rham functor (see \cite[Theorem 9.4.8]{DK16} and the sheafification functor (see \cite[Proposition 3.13]{DK21sh}), and duality interchanges $\EE\fj_*$ with $\EE \fj_{!!}$ and $\EE j^{-1}$ with $\EE j^!$ (see \cite[Proposition 3.3.3]{DK19}). Note that $\EE j^{-1}\iso \EE j^!$ since $j$ is an open embedding. Because all the objects involved are $\R$-constructible (and hence duality is an involution on these objects), this concludes the proof.
\end{proof}

\subsection{Duality pairings for algebraic connections}\label{subsec:dualityBEH}

A well-known result of S.\ Bloch and H.\ Esnault (in dimension one, see \cite{BE04}) and M.\ Hien (for dimensions greater than one, see \cite{Hie09} and also \cite{Hie07}) shows that, for a flat algebraic connection $(E,\nabla)$ on a smooth quasi-projective variety $U$, there is a perfect pairing between algebraic de Rham cohomology with coefficients in $(E,\nabla)$ and rapid decay homology cycles, inducing a duality between the respective (co-)homology spaces (see \ref{eqn:integralpairing} for the precise statement). Concretely, the pairing can be realized as an integration, or period, pairing of differential forms $\alpha \otimes e$ with coefficients in $E$ and of chains $\gamma \otimes \epsilon^\vee$ with coefficients in $E^\vee$ having rapid decay along the boundary:
\begin{equation}\label{eqn:periodpairing}
( \alpha \otimes e, \gamma \otimes \epsilon^\vee) \mapsto \int_\gamma \langle e,\epsilon^\vee \rangle \, \alpha
\end{equation}
where $\langle e,\epsilon^\vee\rangle$ denotes the natural contraction $E \otimes_{\cO_U} E^\vee \to \mathcal{O}_{U}$.

\begin{prop}\label{cor:ModRdpairing1}
Let $(E,\nabla)$ be a flat (algebraic) connection on a smooth quasi-projective variety $U$ over $\C$, and let $(E^\vee,\nabla^\vee)$ be the dual connection on $U$. Then there is a perfect pairing of finite-dimensional $\C$-vector spaces
\begin{equation}\label{eqn:integralpairing}
    \mathrm{H}_{\mathrm{dR}}^\ell\big(U;(E,\nabla)\big) \otimes \mathrm{H}_\ell^{\mathrm{rd}}\big(U;(E^\vee,\nabla^\vee)\big) \to \C
\end{equation}
where $\mathrm{H}_{\mathrm{dR}}$ denotes algebraic de Rham cohomology, and $\mathrm{H}^{\mathrm{rd}}$ denotes rapid decay homology (see \cite{BE04} and \cite{Hie09}).
\end{prop}

\begin{proof}
Recall from \cite{Hie09} that we can describe these de Rham cohomology groups and rapid decay homology groups as the hypercohomology groups of complexes of sheaves on the real blow-up space: We have
$$\HdR^\ell\big(U;(E,\nabla)\big) \iso \hyp^\ell\big(\XD;\DRmodD(\cM)\big)$$ (see \cite{Sab00}) and
$$
\Hrd_\ell\big(U;(E^\vee,\nabla^\vee)\big) \iso \hyp^{-\ell}\big(\XD;\DRrdD(\cM^\vee)\big).
$$
Here, $(X,D)$ is a good compactification of $U$, i.e., $D=X\setminus U$ has normal crossings and $(E,\nabla)$ admits a good formal structure with respect to $(X,D)$. (The existence of such a good compactification was conjectured and proved in certain cases in \cite{Sab00}, and shown in general in \cite{Ked10,Ked11} and \cite{Moc09,Moc11}.) Furthermore, $\cM$ is the (analytic) meromorphic connection associated to the algebraic connection $(E,\nabla)$. In other words, if $\cE$ is the algebraic $\cD_U$-module defined by the connection $(E,\nabla)$ and $j\colon U\hookrightarrow X$ is the inclusion, then $\cM=(j_*\cE)^\mathrm{an}$ is the associated analytic $\cD_X$-module and satisfies $\cM(*D)\iso \cM$. Moreover, $\cM^\vee$ denotes the dual meromorphic connection, which is the meromorphic connection associated to $(E^\vee,\nabla^\vee)$ (or, equivalently, $\cM^\vee=(\mathbb{D}_X\cM)(*D)$).

With these identifications, the claim follows after noting 
$$
\DRmod(\cM) \iso \big(\DRE(\cM)\big)^{\mathrm{mod}\, D},
$$
and
$$
\DRrdD(\cM^\vee) \iso \big(\DRE(\cM^\vee)\big)^{\mathrm{rd}\, D} \iso \big(\DRE(\mathbb{D}_X \cM)\big)^{\mathrm{rd}\, D} \iso \big(\DE(\DRE(\cM))\big)^{\mathrm{rd}\, D},
$$
and then applying (a version on $\XD$ of) Proposition \ref{prop:duality3}. (Note that $\XD$ is compact.)
\end{proof}

\vspace{1cm}


\subsection*{Acknowledgements}
First of all, we would like to thank Claude Sabbah who pointed out many questions studied in this article to us, and in particular inspired us with his recent work on nearby cycles for holonomic $\cD$-modules in the higher-dimensional case. We thank him and Andrea D'Agnolo for useful discussions and correspondence during the preparation of this work. Moreover, we are grateful to Pierre Schapira for explanations on the duality between Whitney functions and tempered distributions, which helped us in proving the various duality statements in Section~\ref{sec:TWduality}. We also want to thank Marco Hien for his interest in our work and for encouraging us to make explicit the connection between the different dualities in this context. The research of Andreas Hohl is funded by the Deutsche Forschungsgemeinschaft (DFG, German Research Foundation), Projektnummer 465657531.

\end{document}